\documentclass{article}
\usepackage[utf8]{inputenc}
\usepackage[margin=1in]{geometry}
\usepackage{amsfonts}
\usepackage{float}
\usepackage{amssymb}
\usepackage{amsmath}
\usepackage{natbib}

\usepackage{graphicx}
\usepackage{url}
\usepackage[normalem]{ulem}
\usepackage{tikz}
\usepackage{color}

\usetikzlibrary{cd}
\tikzset{commutative diagrams/diagrams={ampersand replacement=\&}}

\usepackage[T1]{fontenc}
\usepackage[framed,numbered]{matlab-prettifier}

\usepackage{relsize}
\usepackage{amsmath,amscd}
\usepackage[margin=1in]{geometry}
\usepackage{mathtools}
\usepackage{amsthm}
\definecolor{darkblue}{rgb}{0.0, 0.0, 0.8}
\definecolor{darkred}{rgb}{0.8, 0.0, 0.0}
\definecolor{darkgreen}{rgb}{0.0, 0.8, 0.0}

\makeatletter
\let\original@bibitem\bibitem
\def\bibitem#1#2\par{%
  \@ifundefined{cited@#1}{}{\original@bibitem{#1}#2\par}%
}
\renewcommand
\makeatother

\theoremstyle{definition}
\newtheorem{theorem}{Theorem}[section] 

\newtheorem{lemma}[theorem]{Lemma}

\newtheorem{proposition}[theorem]{Proposition}

\newtheorem{example}[theorem]{Example}
\newtheorem{remark}[theorem]{Remark}
\newtheorem{definition}[theorem]{Definition}

\newcommand{\E}{\mathbb{E}}
\newcommand{\R}{\mathbb{R}}
\newcommand{\V}{\mathbb{V}}
\newcommand{\X}{\mathbb{X}}
\newcommand{\Y}{\mathbb{Y}}
\newcommand{\T}{\mathbb{T}}
\newcommand{\U}{\mathbb{U}}

\newcommand{\CC}{\mathcal{C}}
\newcommand{\DD}{\mathcal{D}}

\newcommand{\PP}{\mathcal{P}}

\renewcommand{\SS}{\mathcal{S}}

\newcommand{\TT}{\mathcal{T}}

\newcommand{\Ifunc}{\mathbf{I}}

\newcommand{\Ffunc}{\mathbf{F}}
\newcommand{\Gfunc}{\mathbf{G}}

\newcommand{\Vfunc}{\mathbf{V}}

\newcommand{\Int}{\mathbf{Int}}

\newcommand{\Set}{\mathbf{Set}}

\newcommand{\Groups}{\mathbf{Groups}}
\newcommand{\Ab}{\mathbf{Ab}}
\newcommand{\vect}{\mathbf{vect}}
\newcommand{\Mor}{\mathbf{Mor}}
\newcommand{\Ob}{\mathbf{Ob}}

\newcommand{\Top}{\mathbf{Top}}

\newcommand{\Reeb}{\mathbf{Reeb}}

\providecommand{\KEY}[1]
{
  \small	
  \textbf{Key words.} #1
}

\date{}

\begin{document}
\title{Tree decomposition of Reeb graphs, parametrized complexity, and applications to phylogenetics}
\maketitle

\author{\centerline{Anastasios Stefanou
    \footnote{stefanou.3@osu.edu; 614-688-3198\\
ORCID id: https://orcid.org/0000-0002-5408-9317\\
        Mathematical Biosciences Institute; Department of Mathematics \\
        The Ohio State University}
}}

\begin{abstract}

Inspired by the interval decomposition of persistence modules and the extended Newick format of phylogenetic networks, we show that, inside the larger category of \textit{ordered Reeb graphs}, every Reeb graph with $n$ leaves and first Betti number $s$, is equal to a coproduct of at most $2^s$ trees with $(n + s)$ leaves. 
Reeb graphs are therefore classified up to isomorphism by their tree decomposition.
An implication of this result, is that the isomorphism problem for Reeb graphs is fixed parameter tractable when the parameter is the first Betti number. 
We propose ordered Reeb graphs as a model for time consistent phylogenetic networks and propose a certain Hausdorff distance as a metric on these structures.
\end{abstract}

\noindent\KEY{Coproducts, decomposition, complexity, Betti number, Reeb graphs, phylogenetic networks.}

\section*{Acknowledgements}
A. Stefanou was partially supported by the National Science Foundation through the grant NSF-CCF-1740761 TRIPODS TGDA@OSU, and also by the grant NSF DMS-1440386 Mathematical Biosciences Institute, at The Ohio State University.
The author gratefully thanks two anonymous reviewers whose feedback significantly increased the quality of the manuscript.
Furthermore, the author thankfully acknowledge F.~Mémoli, E.~Munch, J.~Curry, S.~Kurtek, W.~KhudaBukhsh, and  A.~Foroughipour for many helpful discussions during the course of this work.

\section{Introduction}
Reeb graphs encode the evolution of connected components of a space $\X$ along a real valued map $f$ on $\X$ \cite{reeb1946points}.
Originated from Morse theory, Reeb graphs have been of particular interest to the fields of computational geometry  \cite{cohen2009extending}, \cite{agarwal2006extreme}, \cite{harvey2010randomized} ,\cite{di2016edit}, and computational topology  \cite{morozov2013interleaving}, \cite{edelsbrunner2008reeb}, \cite{cole2004loops}, \cite{bauer_et_al:LIPIcs:2015:5146}, and they have found a plethora of applications in computer graphics and computer science \cite{escolano2013complexity}, \cite{hilaga2001topology}, \cite{chazal2013gromov}, \cite{ge2011data}, \cite{dey2013efficient}, \cite{wood2004removing}. 
See  \cite{biasotti2008reeb} for a survey.
One variation of Reeb graphs that has recently been proposed is Mapper \cite{singh2007topological} which has been quite successful on big data sets \cite{nicolau2011topology}, \cite{yao2009topological}.

\subsection{Related work}
de Silva et al.~(2016) showed that any Reeb graph can be identified with a constructible $\Set$-valued cosheaf on $\R$ \cite{de2016categorified}. 
Thus, Reeb graphs can be thought of as generalized persistence modules in the setting of Bubenik et al.~(2015) \cite{bubenik2015metrics}.
A \textit{generalized persistence module} is any functor $\Ffunc: \PP\to\CC$ from a poset $\PP$ to a category $\CC$ \cite{bubenik2015metrics}.
When $\PP=(\R ,\leq)$ the poset of real numbers and $\CC=\vect_k$ is the category of finite dimensional $k$-vector spaces, we obtain the notion of a \textit{pointwise finite dimensional (p.f.d.) persistence module}.
Crawley-Boevey (2015) has shown that every p.f.d.~persistence module $\Vfunc:(\R ,\leq)\to\vect_k$ decomposes into a direct sum of interval persistence modules \cite{crawley2015decomposition}.
The multiset of intervals $B(\Vfunc)$ associated to $\Vfunc$ is called \textit{the barcode of $\Vfunc$} \cite{carlsson2005persistence}.
Because of this decomposition, one can easily check that the isomorphism complexity of p.f.d.~persistence modules is polynomial.

However, this is no longer true for arbitrary generalized persistence modules.
Bjerkevik et al.~(2018) has shown that the isomorphism complexity of Reeb graphs is GI-complete \cite{bjerkevik_et_al:LIPIcs:2018:8726}, namely deciding if two Reeb graphs are isomorphic it is at least as hard as the graph isomorphism problem.
However the graph isomorphism problem has shown to be fixed parameter tractable with respect to several parameters, such as: tree-distance
width \cite{yamazaki1997isomorphism}, tree-depth \cite{bouland2012tractable}, and tree-width \cite{lokshtanov2017fixed}.
Hence it is natural to wonder whether Reeb graphs, like graphs, are fixed parameter tractable with respect 
to some topologically meaningful parameter.

On the other hand, we can think of any Reeb graph as a weighted directed acyclic graph.
Weighted directed acyclic graphs  are used as the main method for modelling phylogenetic trees and networks \cite{cardona2013cophenetic}, \cite{billera2001geometry}, \cite{semple2003phylogenetics}, \cite{huson2010phylogenetic}.
The isomorphism classes of phylogenetic trees are in one to one correspondence with nested parentheses, a method known today as the \textit{Newick format}, which was already noticed by A.~Cayley (1857).
For general phylogenetic networks, Cardona et al.~(2008) proposed a variant of the Newick format, called the \textit{extended Newick format} \cite{cardona2008extended}.
The idea is: given a fixed ordering on the children nodes of a rooted phylogenetic network with $n$-labelled leaves and $s$ reticulations (Betti number), we can represent that network as a phylogenetic tree with $(n + s)$-labelled leaves where some of the leaves are allowed to have repeated nodes. 
A.~Dress (2007) proposed a categorical approach to view phylogenetic networks called \textit{$X$-nets} \cite{dress2007category}. 

\subsection{Our contribution}
Inspired by the decomposition of p.f.d.~persistence modules into interval persistence modules, we show that any Reeb graph decomposes into a coproduct of trees. 
The construction of each of these trees is an analogue of the extended Newick format in the setting of arbitrary Reeb graphs (not necessary rooted).
First, in Sec.~2 and 3 we mention the basic definitions and tools from category theory \cite{mac2013categories}, and the setting of Reeb graphs as studied in \cite{de2016categorified}, which we need in order to formulate properly the tree decomposition of Reeb graphs.
In Sec 4 we show that the category of  Reeb graphs with a fixed edge structure forms a thin category inside of which every Reeb graph decomposes into a coproduct of trees.
As an implication of this decomposition, we show that Reeb graph isomorphism is fixed parameter tractable where the parameter is the first Betti number.

\section{Categorical structures}
Category theory is fundamentally a language that formalizes mathematical structure having the capability of bridging together different mathematical constructions or theories.
In this section we give the basic definitions and tools from category theory that we need.

\subsection{Basic definitions}
Category theory is a general theory of functions.
A general notion of a function is called a \textit{morphism} and the notion of a set is replaced by an \textit{object}.
An object can be any mathematical construction and its not necessary to be a set.
In contrast with set theory the focus is concentrated in the study of morphisms between objects rather than just study the objects themselves.
In particular, we require that morphisms between objects to have a composition operation that is associative and unital.
The structure we obtain is said to be a \textit{category}.
A good source for an introduction to category theory is \cite{mac2013categories}.

First we define the notion of a category.
Here by a \textit{class} we mean a collection of sets that is unambiguously defined by property that all these sets share in common.
A class might not be a set and if that is the case is called a \textit{proper class}.

\begin{definition}
A \textit{category} $\CC$ consists of
\begin{itemize}
    \item a class $\Ob\CC$ whose elements $X,Y,\ldots$ are called \textit{objects}, together with
    \item for each pair of objects $X,Y$ in $\CC$ a set $\Mor_{\CC}(X,Y)$, whose elements are called \textit{morphisms} and denoted by $f:X\to Y$, and each having a unique source $X$ and a unique target $Y$,
    \item for each object $X$ in $\CC$ an \textit{identity} morphism $I_X:X\to X$,
    \item a binary operation $\circ: \Mor_{\CC}(X,Y)\times\Mor_{\CC}(Y,Z)\to\Mor_{\CC}(X,Z)$, $(f,g)\mapsto g\circ f$
    called \textit{composition} which is associative and unital, i.e.
        \begin{align*}
        h\circ(g\circ f)&=(h\circ g)\circ f\\
        f\circ I_X&=I_Y\circ f
    \end{align*}
    for any triple of morphisms $f:X\to Y$, $g:Y\to Z$ and $h:Z\to W$ in $\CC$.
\end{itemize}
\end{definition}
\begin{definition}
A morphism $f:X\to Y$ is said to be an \textit{isomorphism} from $X$ to $Y$ if there exists a morphism $g:Y\to X$ (often called the inverse) such that $g\circ f=I_X$ and $f\circ g=I_Y$.
Two objects are said to be isomorphic if there exists an isomorphism from $X$ to $Y$.
\end{definition}

\begin{example}
Examples of  categories include:
\begin{itemize}
\item the category $\Set$  whose objects are sets and morphisms are  functions between sets
\item the category $\Top$ whose objects are topological spaces and morphisms are continuous maps.
\item the category $\Groups$ whose objects are groups and morphisms are group homomorphisms
\item the category $\Ab$ whose objects are abelian groups and morphisms are group homomorphisms
\end{itemize}
\end{example}
\begin{definition}
A category $\DD$ whose objects and morphisms are in $\CC$ and with the same identities and composition operation as of $\CC$ is said to be a \textit{subcategory of $\CC$}.
\end{definition}
\noindent Let $\CC$ be a category and let $\SS$ be any subset of $\Ob\CC$.
Then we can consider the same sets $\Mor_\CC(X,Y)$ of morphisms between $X,Y \in \SS$.
That way we obtain a category with the same morphisms but fewer objects.
We say that $\SS$ forms a \textit{full subcategory} of $\CC$.
In a full subcategory we only need to specify what are the objects so we often say $\SS$ is \textit{the full subcategory of $\CC$ whose objects are in $\SS$}.
For example the category $\Ab$
is the full subcategory of $\Groups$ whose objects are abelian groups.

\begin{example}
The type of categories we work on are the following:
\begin{itemize}
\item \textit{slice categories}: given a category $\CC$ and an object $X$ we consider the slice category $\CC\downarrow X$ whose objects are tuples $(Y,f)$ where $Y\in\Ob\CC$ and $f\in\Mor_{\CC}(Y,X)$, and morphisms $\varphi:(Y,f)\to(Z,g)$ are ordinary morphisms $\varphi:Y\to Z$ in $\CC$ such that $g\circ\varphi =f$.
\item \textit{thin categories}: a category $\CC$ is called thin if for every pair of objects $X,Y$ in $\CC$ there exists at most one morphism $f:X\to Y$ in $\CC$.
When a morphism $f:X\to Y$ exists we write $X\leq Y$.
A thin category coincides with the notion of a \textit{preorder}.
\end{itemize}
\end{example}

\noindent Now, we define the notion of maps that preserve the structure of a category.
\begin{definition}
A \textit{functor} $\Ffunc:\CC\to\DD$ between categories consists of
\begin{itemize}
    \item a function $\Ffunc:\Ob\CC\to\Ob\DD$, $X\mapsto \Ffunc(X)$, together with
    \item for each pair of objects $X,Y$ in $\CC$, a function
    \begin{align*}
    \Ffunc_{X,Y}:\Mor_{\CC}(X,Y)&\to\Mor_{\DD}(\Ffunc(X),\Ffunc(Z))\\
    f&\mapsto\Ffunc[f]
    \end{align*}
    such that  for any object $X$ and any morphisms $f:X\to Y$, $g:Y\to Z$ in $\CC$:
    \begin{align*}
        \Ffunc[g\circ f]&=\Ffunc[g]\circ\Ffunc[f]\\
        \Ffunc[I_{X}]&=I_{\Ffunc(X)}.
    \end{align*}
\end{itemize}
\end{definition}
\noindent When $\CC=\DD$, $\Ffunc$ is called an \textit{endofunctor}.
A special case is the identity endofunctor $\Ifunc_{\CC}:\CC\to\CC$ that sends each object and morphism to itself.

\vspace{1em}
\noindent The collection of all functors from a category $\CC$ to a category $\DD$ forms a category on its own called a \textit{functor category} and it is denoted by $[\CC,\DD]$: the objects are functors $\Ffunc:\CC\to\DD$ and the morphisms are natural transformations $\eta:\Ffunc\Rightarrow\Gfunc$. \begin{definition}
A \textit{natural transformation} $\eta:\Ffunc\Rightarrow\Gfunc$ consists of a family of morphisms $\eta_X:\Ffunc(X)\to\Gfunc(X)$ in $\DD$ one for each object $X$ in $\CC$, such that the diagram
\[
\begin{tikzcd}
\Ffunc(X)\arrow[r,"{\eta_X}"]\arrow[d,swap,"{\Ffunc[f]}"]\&\Gfunc(X)\arrow[d,"{\Gfunc[f]}"]\\
\Ffunc(Y)\arrow[r,"{\eta_Y}"]\&\Gfunc(Y)
\end{tikzcd}
\]
commutes for every morphism $f:X\to Y$ in $\CC$.
In the special case where each $\eta_X$ is an isomorphism in $\DD$, then $\eta$ is said to be a \textit{natural isomorphism} and we write $\Ffunc\cong \Gfunc$.
\end{definition}
Every time we write $\Ffunc\cong \Gfunc$ we mean there exists a natural isomorphism $\eta:\Ffunc\Rightarrow\Gfunc$.

\begin{definition}
A pair of categories $\CC,\DD$ are said to be \textit{equivalent} if there exist functors $\Ffunc:\CC\to\DD$ and $\Gfunc:\DD\to\CC$ such that $\Ffunc\circ\Gfunc\cong \Ifunc_\DD$ and $\Gfunc\circ\Ffunc\cong \Ifunc_\CC$.
In the special case where $\Ffunc\circ\Gfunc= \Ifunc_\DD$ and $\Gfunc\circ\Ffunc= \Ifunc_\CC$, the categories $\CC$ and $\DD$ are said to be \textit{isomorphic}.
\end{definition}


\subsection{Coproducts}
We define the notion of a coproduct of objects in a category $\CC$.
This is the dual notion of a product \cite{mac2013categories}.
However, here we focus only on the definition of coproducts since this is the only notion we use.
\begin{definition}
Let $X_1,\ldots,X_n$ be objects in $\CC$.
An object is called \textit{the coproduct} of $X_1,X_2,\ldots,X_n$, written $\coprod X_i$, if there exist morphisms $\iota_j:X_j\to \coprod X_i$,  $j=1,2,\ldots,n$ satisfying the following \textit{universal property}: for any object $Y$ and any pair of morphisms $f_j:X_j\to Y$, $j=1,2,\ldots,n$, there exists a unique morphism $f:\coprod X_i\to Y$ such that the diagrams
\[
\begin{tikzcd}
X_j\arrow[r, "{\iota_j}"] \arrow[rd,swap,"{f_j}"]\& \coprod X_i\arrow[d,"{f}"]
\\
\& Y
\end{tikzcd}
\]
commute for $j=1,2,\ldots,n$.
\end{definition}
Note that by the universal property of coproducts, the morphisms $\iota_j$ are uniquely defined up to a unique natural isomorphism.
The morphisms $\iota_j:X_j\to \coprod X_i$ are called \textit{coprojections}.

\begin{definition}
An object $X$ is said to be \textit{decomposable} if it is isomorphic to a coproduct of $n$ objects in $\CC$ where $n\geq2$.
Otherwise $X$ is said to be \textit{indecomposable}.
\end{definition}

\begin{example}
Here we give some basic examples of categorical coproducts.
\begin{itemize}
    \item If $\CC$ is the category of all sets $\Set$, then the coproduct is given by the disjoint union $\coprod$ of sets.
    \item If $\CC$ is the category of groups $\Groups$ then the coproduct is given by the free product $*$ of groups.
    \item If $\CC$ is the category of abelian groups $\Ab$ then the coproduct is the direct sum $\oplus$ of abelian groups.
\end{itemize}
\end{example}

\section{Combinatorial structures}
In this section we consider the setting of Reeb graphs as developed by V.~de Silva et al.~(2016) \cite{de2016categorified}.
We define Reeb graphs and examine how they relate to directed acyclic graphs.

\subsection{Reeb graphs}
The main tool we use to visualize relationships among objects is a \textit{graph}.
A graph $G=(\V,\E)$ consists of a collection $\V$  of objects called \textit{vertices}, e.g. $v_1,v_2,\ldots$ and a set $\E$ of connections $e_1,e_2,\ldots$ between vertices called \textit{edges}.

Generally we can define a Reeb graph as a connected graph $\X$ together with a real valued map $f:\X\to\R$ which is strictly monotone when restricted to edges.
However with this definition we are not making precise the exact way $\X$ can be constructed in conjunction with the map $f$ being monotone restricted to edges.
Making this more precise is what we do in this paragraph.

First we need to talk about the general setting of $\R$-spaces.
An \textit{$\R$-space} $(\X,f)$ is a space $\X$ together with a real valued continuous map $f:\X\to\R$.
A \textit{morphism} of $\R$-spaces $(\X,f),(\Y,g)$--also called a \textit{function preserving map}--is an ordinary continuous map $\varphi:\X\to\Y$ such that $g\circ\varphi=f$.
The collection of these objects forms the slice category $\Top\downarrow\R$.
Now let us return to Reeb graphs.
\begin{definition}
\label{def:Reeb def}
A connected $\R$-space $(\X,f)$ is said to be a \textit{Reeb graph} if
it is constructed by the following procedure, which we call \textit{a structure on $(\X,f)$}:

 Let $S=\{a_1<\ldots<a_{k}\}$ be an ordered subset of $\R$ with $k\geq2$.

\begin{itemize}
    \item For each $i=1,\ldots,k$ we specify a non-empty set $\V_i$ of vertices which lie over $a_i$,
    \item For each $i=1,\ldots,k-1$ we specify a non-empty set of edges $\E_i$ which lie over $[a_i,a_{i+1}]$
    \item For $i = 1,\ldots,k-1$, we specify a `down  map' $D_i:\E_i \to \V_i$
    \item For $i = 1,\ldots,k-1$, we specify an `upper  map' $U_i: \E_i \to \V_{i+1}$.
\end{itemize}
The space $\X$ is the quotient $\U/\sim$ of the disjoint union
$$\U=\biggl(\coprod_{i=1}^{k}(\V_i\times\{a_i\})\biggl)\coprod\biggl(\coprod_{i=1}^{k-1}(\E_i\times [a_i,a_{i+1}])\biggl)$$
with respect to the identifications $(D_i(e),a_i)\sim(e,a_i)$ and $(U_i(e),a_{i+1})\sim(e,a_{i+1})$, for all $i=1,\ldots,k-1$, with the map $f$ being the projection onto the second factor.
\end{definition}
\begin{remark}
Note that  to every Reeb graph corresponds a unique minimal set $S$, known as the \textit{critical set} of $\X$ \cite{de2016categorified}.
We consider a definition for arbitrary set $S$ because this allows for more flexibility: we can describe the morphisms between Reeb graphs easily if we consider a common set $S$ for both Reeb graphs, e.g.~by considering the union of their two $S$-sets.
\end{remark}
The set  $\V=\coprod_{i=1}^{k}\V_i$ is said to be a \textit{vertex-set} for $\X$ and the set $\E=\coprod_{i=1}^{k-1}\E_i$ is said to be a \textit{edge-set} for $\X$\footnote{the reason we call these sets in this way, e.g.~we say `a vertex-set' instead of `the vertex-set', is because they depend on the choice of $S$. }.
If we forget the map $f$ associated to a Reeb graph, then topologically $\X$ forms a graph $(\V,\E)$ on its own.
See Fig.~1~for an example of a Reeb graph.
\begin{definition}
A \textit{morphism} of Reeb graphs $(\X,f)$ and $(\Y,g)$ is any morphism $\varphi:(\X,f)\to(\Y,g)$ between these $\R$-spaces.
\end{definition}
Thus, the collection of all  Reeb graphs forms a full subcategory $\mathbf{Reeb}$  of $\Top\downarrow\R$.
As shown in \cite{de2016categorified} Reeb graphs can be identified with constructible $\Set$-valued cosheaves on $\R$.
This equivalence of categories allows for one to consider the following combinatorial description of the morphisms of  Reeb graphs.
\begin{proposition}[Prop~3.12 in \cite{de2016categorified}]
\label{prop:consistency}
Let $(\X,f),(\Y,g)$ be a pair of Reeb graphs with a common  set $S=\{a_1<\ldots<a_k\}$, let $\V^{\X}$, $\V^{\Y}$, and $\E^{\X}$, $\E^{\Y}$ be their vertex-sets and edge-sets respectively.
Any function preserving map $\varphi:(\X,f)\to(\Y,g)$ of  Reeb graphs is completely determined by 
\begin{itemize}
    \item Functions $\varphi_{i}^{\V}:\V_i^{\X}\to \V_i^{\Y}$
    \item Functions $\varphi_{i}^{\E}:\E_i^{\X}\to\E_i^{\Y}$, satisfying the
    \item Consistency conditions: $\varphi_{i}^{\V}D_i^{\X}=D_i^{\Y}\varphi_i^{\E}$ and $\varphi_{i+1}^{\V}U_i^{\X}=U_{i}^{\Y}\varphi_i^{\E}$ for all $1\leq i\leq k-1$.
\end{itemize}
\end{proposition}
Any function preserving map $\varphi:(\X,f)\to(\Y,g)$, 
since $f$ and $g$ are by definition the projections to the second coordinate of $\X$ and $\Y$ respectively, is given by 
\begin{align*}
    \varphi:(\X,f)&\to(\Y,g)\\
    [(v,t)]&\mapsto [(\varphi_i^{\V}(v),t)]\text{, for all  }v\in \V_i^\X\text{, for all }i=1,\ldots,k.\\
 [(e,t)]&\mapsto [(\varphi_i^{\E}(e),t)]\text{, for all  }e\in \E_i^\X\text{, for all }i=1\ldots,k-1.\footnote{the bracket $[\text{ }]$ denotes an equivalence classes in the quotient space $\X=\U/\sim$.}
\end{align*}

\begin{figure}[h!]
\label{fig:definition}
  \begin{center}
      \includegraphics[width=1.2\textwidth]{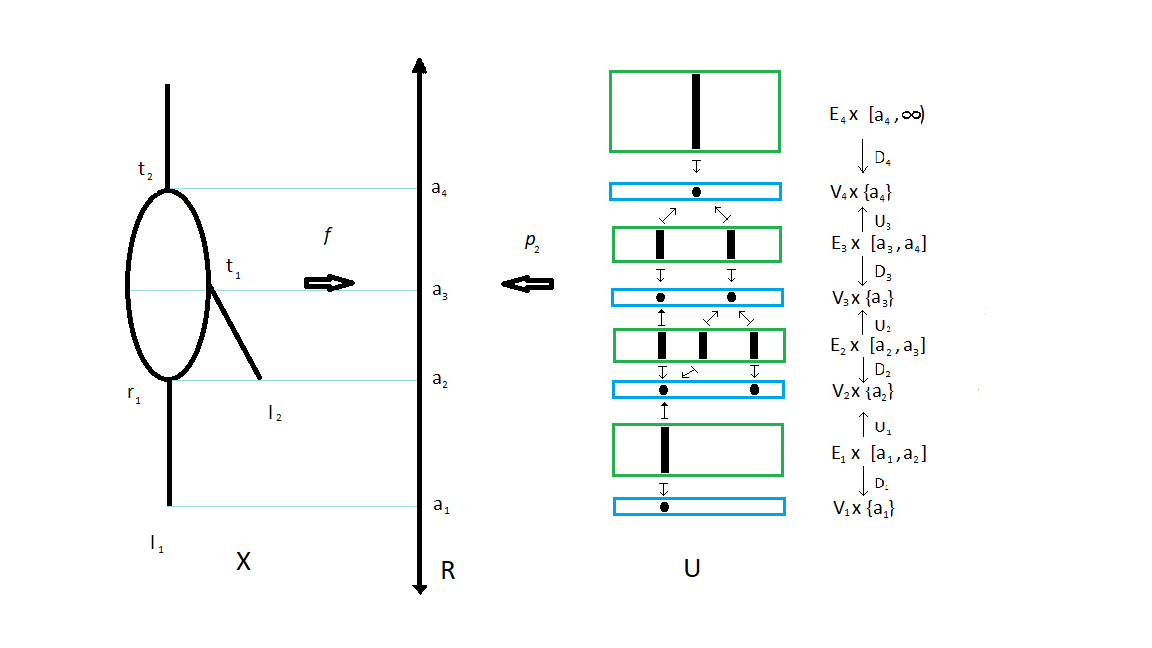}
      \caption{An example of a Reeb graph}
      \end{center}
\end{figure}

\subsection{Reeb graphs viewed as directed acyclic graphs}
\label{sec:vertex-weighted}
If we allow the edges $e$ of a graph $G$ to have a direction, i.e. $x\xrightarrow{e} y$ then the resulting graph is said to be a \textit{digraph} and the edges are called \textit{directed edges} or \textit{arrows}. 
A \textit{directed path of length $n$} on a digraph is a sequence of arrows $x_0\to x_1\to x_2\to x_3\to\ldots\to x_{n-1}\to x_n$ in $G$.
A directed path  that starts and end at the same vertex is called a \textit{directed cycle}.
A digraph with no directed cycles is said to be a \textit{directed acyclic graph (DAG)}.

With the map $f:\X\to\R$, we give a direction to each edge $e$ connecting $x_1$ and $x_2$ in $\X$, by declaring $e:x_1\to x_2$ whenever $f(x_1)> f(x_2)$. 
That way, the underlying graph $(\V,\E)$ of a Reeb graph $\X$ obtains the structure of a directed acyclic graph.
Furthermore, each vertex $v$ of $\X$ receives a real weight  $w(v):=f(v)$ via $f$.
So, every Reeb graph can be thought of as a vertex-weighted DAG.
\begin{remark}
\label{rem:weighted}
Note, in particular, for Reeb graphs, because of the vertex-weight $w(v):=f(v)$ each directed edge $e:x_1\to x_2$ receives a strictly positive weight  $w(e):=f(x_1)-f(x_2)>0$, where $f(x_1)=a_{i+1}$ and $f(x_2)=a_{i}$, for some $i=1,\ldots,k-1$, are consecutive critical numbers of $(\X,f)$.
\end{remark}

Let $(\X,f)$ be a  Reeb graph with critical set $S=\{a_1,\ldots,a_k\}$. 
Let $x$ be a vertex of $\X$ with $f(x)=a_i$.
Then an edge $e\in D_i^{-1}(x)$ ($e\in U_{i-1}^{-1}$) is said to be an edge incident from above (below) of $x$.
Also the number of edges $\mathbf{indeg}(x)=|D_i^{-1}(x)|$ that are incident from above of $x$ is called the \textit{indegree of $x$}.
Similarly, the number of edges $\mathbf{outdeg}(x)=|U_{i-1}^{-1}(x)|$ is called the \textit{outdegree of $x$}.

There are three cases that can happen for a node $x$:
\begin{itemize}
 \item If $\mathbf{indeg}(x)\leq 1$, then $x$ is said to be a \textit{tree-vertex}.
 \item If $\mathbf{indeg}(x)\geq 2$ and with $x$ a cycle in the Reeb graph (viewed as a DAG) closes ($x$ is the `bottom' of a cycle), then $x$ is said to be a \textit{reticulation-vertex}.
 \item If $\mathbf{indeg}(x)\geq 2$ and with $x$ no cycles closes in the Reeb graph, then $x$ again is said to be a \textit{tree-vertex}.
\end{itemize}
A vertex $x$ such that $\mathbf{indeg}(x)=\mathbf{outdeg}(x)=1$ is called \textit{regular}.
 We denote by $T(\X)$ the set of tree-vertices of $\X$, and by $R(\X)$ the set of all reticulation-vertices of $\X$.
If ($\mathbf{outdeg}(x)=0$) or ($\mathbf{outdeg}(x)=1$ and $\mathbf{indeg}(x)=0$), then $x$ is said to be a \textit{leaf}.
By definition, it may happen that a leaf is also a reticulation-vertex (i.e.~$\mathbf{outdeg}(x)=0$ and $\mathbf{indeg}(x)\geq2$).
\begin{remark}
The reason we consider this definition of leaves is so that the tree decomposition of Reeb graphs to work for this type of reticulation vertices. 
This would be more clear in Ex.~\ref{ex:3}.
\end{remark}
We denote by $L(\X)$ the set of all leaves of $\X$.
Then, $L(\X)\subset T(\X)$.

Viewing Reeb graphs as DAGs  is also important for computing the Betti number $s=B_1(\X)$ of a Reeb graph $\X$.
The Betti number counts the minimum number of cycles of a graph that generate all possible cycles in the graph.
By definition, a Reeb graph $(\X,f)$ is a \textit{tree} if and only if it has no reticulation-vertices (and thus, no cycles).
Since $\X$ is always connected by definition, the Euler characteristic provides the formula
$$s=|\E|-|\V|+1,$$
Furthermore, from the theory of directed graphs we have the \textit{degree sum formula}
$$\sum_{x\in \V}\mathbf{indeg}(x)=\sum_{x\in \V}\mathbf{outdeg}(x)=|\E|$$
By combining these two equations we get
$$s= 1+\sum_{x\in \V}(\mathbf{indeg}(x)-1).$$
\begin{proposition}
Let $r_1,\ldots,r_m$ be the reticulation-vertices of $\X$ with indegrees $d_1,\ldots,d_m\geq 2$ respectively.
Then 
$$s=\sum_{i=1}^{m}(d_i-1).$$
\end{proposition}
\label{prop:betti}
\begin{proof}
Thinking of $\X$ as a DAG, then if we remove for each reticulation-vertex $r_i$ all of its incident from above edges except one, from the DAG $\X$, then we get back a  directed tree $\T_\X=(\V^{\T_\X},\E^{\T_\X})$ having the same tree-vertices and their--incident from above--edges as in $\X$ plus another $m$-additional vertices, $r_1,\ldots, r_m$, that now are viewed as tree-vertices, each having indegree $1$.
This implies that $T(\V_\X)\cup\{r_1,\ldots,r_m\}=\V^{T_{\X}}$ and $\sum_{i=1}^{m}(\mathbf{indeg}_{\T_{\X}}(r_i)-1)=0$. 
We compute
\begin{align*}
s&=1+\sum_{x\in \V}(\mathbf{indeg}(x)-1)\\
&=1+\sum_{x\in T(\V)}(\mathbf{indeg}(x)-1)+\sum_{x\in R(\V)}(\mathbf{indeg}(x)-1)\\
&=1+\sum_{x\in T(\V)}(\mathbf{indeg}(x)-1)+0+\sum_{i=1}^{m}(d_i-1)\\
&=1+\sum_{x\in T(\V)}(\mathbf{indeg}(x)-1)+\sum_{i=1}^{m}(\mathbf{indeg}_{\T_{\X}}(r_i)-1)+\sum_{i=1}^{m}(d_i-1)\\
&=1+\sum_{x\in\V^{T_{\X}}}(\mathbf{indeg}(x)-1) +\sum_{i=1}^{m}(d_i-1)\\
&=s^{T_{\X}}+\sum_{i=1}^{m}(d_i-1)\\
&=0+\sum_{i=1}^{m}(d_i-1)\\
&=\sum_{i=1}^{m}(d_i-1).
\end{align*}
\end{proof}
Finally, note that although any Reeb graph can be thought of as a vertex-weighted  DAG, the other direction is not true: not every vertex-weighted DAG is a Reeb graph.
\begin{example}
\label{ex:counter}
Consider the vertex-weighted DAG, $G=(V(G),E(G),w)$,  that has edges $e_1=(x_1,x_2),e_2=(x_2,x_3)$ connecting $x_1$ with $x_2$ and $x_2$ with $x_3$ and an edge $e_3=(x_1,x_3)$ connecting directly $x_1$ with $x_3$ with a single edge, and such that $w(x_1)=2$, and $w(x_2)=w(x_3)=1$.
Then we claim that there exists no  Reeb graph $(\X,f)$ such that $V(\X)=V(G)$ and $f_{|V(G)}=w$.
Indeed assume the contrary there is one such $f$.
Then $f(x_1)=2$ and $f(x_2)=f(x_3)=1$.
By definition of $G$, there exists an edge $e_2:x_2\to x_3$.
By Remk.~\ref{rem:weighted} we get a strictly positive edge-weight $w(e_2)>0$. Howver we compute $w(e_2):=f(x_2)-f(x_3)=w(x_2)-w(x_3)=1-1=0$, a contradiction.
Hence the weighted DAG $G$ cannot be realized as a Reeb graph.
This DAG it is not consistent with `time' in the sense that the edge $e_2=(x_2,x_3)$ represents a change in the nodes from $x_2$ to $x_3$ that happen instantaneously.
In other words the edge $e_2$ is a `horizontal' edge.
Reeb graphs cannot have horizontal edges, from their construction.
\[
\begin{tikzcd}
    \& x_1\arrow[ddr]\arrow[ddl]\\
    \\
    x_2\arrow[rr]\&\& x_3
\end{tikzcd}
\]
\end{example}

\section{Classifying Reeb graphs up to isomorphism}
We 
show that inside a larger category of Reeb graphs, any Reeb graph  is  a coproduct of trees.

\subsection{Ordered Reeb graphs}
\begin{definition}
	An \textit{ordered Reeb graph} $(\X,f,\leq_\X)$ is an ordinary Reeb graph $(\X,f)$ such that its edge sets  and its vertex sets as in Defn.~\ref{def:Reeb def} are in particular partially ordered sets (posets), i.e.~$(\E_i,\leq_{\E_i})$, $i=1,\ldots,k-1$, and $(\V_i,\leq_{\V_i})$, $i=1,\ldots,k$, and also the down maps and upper maps preserve the partial orders, i.e.~$U_i:(\E_i,\leq_{\E_i})\to(\V_{i+1},\leq_{\V_{i+1}})$, and $D_i:(\E_i,\leq_{\E_i})\to(\V_i,\leq_{\V_i})$, for $i=1,\ldots,k-1$.
	The partial orders of the edge posets and vertex posets induce a partial order both on the disjoint union and the quotient space, namely the Reeb graph, denoted by $\leq_\X$.
Indeed, it is known that the category $\mathbf{Pos}$ of posets, just like sets, admit coequalizers and therefore quotients (see Joy of Cats, pg 119, \cite{adamek1990herrlich}). 
\end{definition}
\begin{definition}
	We define the category $\mathbf{Reeb}^{ord}$ whose
	\begin{itemize}
		\item objects $(\X,f,\leq_\X)$ are ordered Reeb graphs.
		\item morphisms $\varphi:(\X,f,\leq_\X)\to(\Y,g,\leq_\Y)$ are ordinary morphisms of Reeb graphs $\varphi:(\X,f)\to(\Y,g)$ that preserve the partial orders of the edge posets and vertex posets, i.e~$\varphi^{\E_i}:(\E^\X_i,\leq_{\E^\X_i})\to(\E^\Y_i,\leq_{\E^\Y_i})$, $i=1,\ldots,k-1$, and $\varphi^{\V_i}:(\V^\X_i,\leq_{\V^\X_i})\to(\V^\Y_i,\leq_{\V^\Y_i})$, $i=1,\ldots,k$, are order preserving maps. 
		\item composition is defined in the obvious way.
	\end{itemize}
\end{definition}
Any finite set $A=\{x_1,\ldots,x_n\}$ can be trivially thought of as a poset $(A,\leq)$ by considering $x_i\leq x_j\Leftrightarrow x_i=x_j$. 
Namely, the only inequalities are the identities.
Thus, any Reeb graph $(\X,f)$ has its edge sets and vertex sets trivially partially ordered, i.e.~the inequalities are only the identities. 
Hence any Reeb graph $(\X,f)$ is an object of $\mathbf{Reeb}^{ord}$. 
That is, $\Ob(\mathbf{Reeb})\subset \Ob(\mathbf{Reeb}^{ord})$.
In particular $\mathbf{Reeb}$ is a full subcategory of $\mathbf{Reeb}^{ord}$. 
Indeed, let $\varphi:(\X,f)\to (\Y,g)$ be a morphism in $\mathbf{Reeb}$.
Then, again trivially we can think of $\varphi$--when restricted to the edge sets and vertex sets respectively--as an order preserving map between  trivially partially ordered sets.
Therefore $\mathbf{Reeb}$ is a full subcategory of $\mathbf{Reeb}^{ord}$.

\subsection{Ordered Reeb graphs with a fixed edge structure}
 Fix an ordered subset $S=\{a_1<\ldots<a_k\}$ of $\R$.
Consider an ordered Reeb graph $(\X,f,\leq_\X)$ with edge posets
$(\E,\leq_\E):=(\E_1,\leq_{\E_1})\coprod \ldots\coprod(\E_{k-1},\leq_{\E_{k-1}}),$ where $\E_i$, $i=1,\ldots,k-1$, as in Defn.~\ref{def:Reeb def}.
We call $\E_\bullet=(\E_1,\ldots,\E_{k-1})$ an \textit{edge sequence} of the ordered  Reeb graph $(\X,f)$ with respect to $S$.

Given a common set $S=\{a_1<\ldots<a_k\}$ for a pair of Reeb graphs $(\X,f,\leq_\X)$ and $(\Y,g,\leq_\Y)$, we say that their edge sequences  $\E_\bullet^\X=(\E^\X_1,\ldots,\E^\X_{k-1})$ and $\E^\Y=(\E^\Y_1,\ldots,\E^\Y_{k-1})$ are \textit{equivalent}, and denote it by $\E_\bullet^\X\cong \E^\Y_\bullet$, if $\E^\X_i\cong \E^\Y_i$, for all $i=1,\ldots,k-1$, as sets.
In other words, $(\X,f,\leq_\X)$ and $(\Y,g,\leq_\Y)$ have equivalent edge sequences if $|\E_i|=|\E'_i|$ for all $i=1,\ldots,k-1$, where $|\cdot|$ denotes the cardinality.


\begin{definition}
Let $(\X,f)$ and $(\Y,g)$ be two Reeb graphs with common set $S$.
We say $(\X,f)$, $(\Y,g)$ have the \textit{same edge structure}, if their edge sequences are equivalent. 
\end{definition}
Note that the relation  `$(\X,f)\sim (\Y,g)$ $\Leftrightarrow$ $(\X,f)$, $(\Y,g)$ have the same edge structure' forms an equivalence relation.
This equivalence relation induces a partition on the objects of $\mathbf{Reeb}$, i.e.~we get
$$\Ob(\mathbf{Reeb})=\coprod_{a\in A}[(\X_a,f_a)],$$
for some index set $A$, where $[]$ denotes the $\sim$-equivalence class.
This fact suggests that each of these blocks can be turned into a category on its own.

Fix an edge-sequence $\E_\bullet$.
\begin{definition}
	We define $\Reeb^{ord}[\E_\bullet]$ the category whose 
\begin{itemize}
\item  objects $(\X,f,\leq_\X,\mu^\X)$ are ordered Reeb graphs $(\X,f,\leq_\X)$ 
together with a family $\mu^{\X}=\{\mu_i^{\X}\}_{i=1}^{k-1}$ of bijections $\mu_i^{\X}:\E_{i}^{\X}\to \E_i$, for all $i=1,\ldots,k-1$, called an $\E$-\textit{edge labelling}, or simply an \textit{edge labelling} if $\E_\bullet$ is given.
\item the morphisms are ordinary morphisms $\varphi:(\X,f,\leq_\X)\to(\Y,g,\leq_\Y)$ in $\Reeb^{ord}$ that preserve the edge-labeling, namely $\mu^{\Y}( \varphi_j^\E(e))=\mu^{\X}(e)$ for all $e\in \E_j^{\X}$ and all $j=1,\ldots,k-1$.
\item composition is defined in the obvious way.
\end{itemize}
\end{definition}
\begin{lemma}
The category $\mathbf{Reeb}^{ord}[\E_\bullet]$ is thin.
\end{lemma}

\begin{proof}
Let $\varphi,\psi:(\X,f,\leq_\X,\mu^\X)\to(\Y,g,\leq_\Y,\mu^\Y)$ be two morphisms in $\mathbf{Reeb}^{ord}[\E_\bullet]$.
By Prop.~\ref{prop:consistency} we have to show that the maps agree at the edge posets and vertex posets.
By definition $\varphi=\psi=(\mu^{\Y})^{-1}\circ(\mu^{\X})$ when restricted to edge posets, because the labellings are bijections.

Let $v$ be a node in $\X$, say $v\in \V_i^{\X}$ for some $i$.
We claim that $\varphi_i^{\V}(v)=\psi_i^{\V}(v)$.
Now, $v$ is either the down image or the upper image of some edge $e\in\E_i^{\X}$, for some $i=1,\ldots,k-1$.
So, we have two cases:
\begin{itemize}
    \item Case 1: $v=D_i^{\X}(e)$.
    By Prop.~\ref{prop:consistency} we have that
    \begin{align*}
        \varphi(v)&=\varphi_i^{\V}(D_i^{\X}(e))\\
        &=D_i^{\Y}(\varphi_i^{\E}(e))\\
        &=D_i^{\Y}((\mu^{\Y})^{-1}(\mu^{\X}(e)))\\
        &=D_i^{\Y}(\psi_i^{\E}(e))\\
        &=\psi_i^{\V}(D_i^{\X}(e))\\
        &=\psi(v).
    \end{align*}
    \item Case 2: $v=U_{i}^{\X}(e)$.
    The proof of case 2 is similar to that of Case 1 and is omitted.
\end{itemize}
\end{proof}

\begin{remark}
\label{rem:partition}
We define the full subcategory $\Reeb[\E_\bullet]$ of $\Reeb^{ord}[\E_\bullet]$ whose objects are Reeb graphs with the trivial partial order (i.e.~ordinary Reeb graphs) and with the same edge structure as $\E_\bullet$.
By definition of $\sim$ we observe that $(\X,f)$ and $(\Y,g)$ have the same edge structure if and only if both $(\X,f)$ and $(\Y,g)$ are objects of $\mathbf{\Reeb}[\E_\bullet]$, for some edge-sequence $\E_\bullet$. 
\end{remark}

\subsection{Tree-decomposition}
\begin{theorem}
\label{thm:ClassReeb}
Fix an edge-sequence $\E_\bullet$.
Let $n,s\geq0$.
Inside the larger category $\Reeb^{ord}[\E_\bullet]$, any Reeb graph $(\X,f,\mu)$ in $\Reeb[\E_\bullet]$, 
with $n$ leaves and Betti number $B_1(\X)=s$, is a coproduct of ordered trees with $(n+s)$-leaves and same $\E$-edge labelling, i.e.
$$(\X,f,\mu)\cong\coprod_{\T\in\TT(\X)}(\T,f_\T,\leq_{\T},\mu),$$
for some set of ordered trees $\TT(\X)$.
\end{theorem}
\begin{proof}
Let $\X=(\X,f,\mu)$ be a Reeb graph in $\Reeb[\E_\bullet]$ with $n$-leaves and $B_1(\X)=s$.
Let $\leq_\X$ be its trivial partial order.
Let $L(\X)=\{l_1,\ldots,l_t\}$, $R(\X)=\{r_1,\ldots,r_m\}$ be its sets of leaves, and reticulation-vertices respectively.
Let $d_i$ be the indegree of the reticulation-vertex $r_i$ for all $i=1,\ldots,m$.
Also for any reticulation-vertex  $r_i$, $i=1,\ldots,m$, let us denote by $\{e_{1}^{(i)},\ldots,e_{d_{i}}^{(i)}\}$ the set of all edges that are incident of $r_i$ from above it.
The basic idea of the proof is to construct a collection of ordered trees with $(n+s)$-leaves, and same edge labelling as $\X$, out of $\X$, 
$$\mathcal{T}(\X)=\{(\T_{(w_1,\ldots,w_m)},f_{(w_1,\ldots,w_m)},\leq_{(w_1,\ldots,w_m)},\mu):1\leq w_i\leq d_i\text{, for all }i=1,\ldots,m\}$$ 
 by breaking up in all possible ways the reticulation-vertices (the bottom of the cycles) of $\X$--without changing the connectivity of the graph--in order to create a tree out of $\X$, by introducing new leaves.   
Consider $\V_j^\X$ and $\E_j^\X$ and the upper and down maps  $U_j^{\X}$ and $D_j^{\X}$, $j=1,\ldots,k-1$ as in Defn.~\ref{def:Reeb def}.
For simplicity of the proof, we denote any $m$-tuple $(w_1,\ldots,w_m)$ by $w$.

Let $w\in \prod_{i=1}^{m}\{1,\ldots,d_i\}$.
We construct an ordered tree $\T=(\T_{w},f_{w},\leq_w,\mu)$ with $(n+s)$-leaves, by changing the structure of $\X$ so as to make an odered tree following the steps below:

\vspace{1em}
\noindent 1. We define 
\begin{itemize}
\item for any $j=1,\ldots,k-1$, the set $\E^{\T_w}_j:=\E^\X_j$ equipped with the trivial partial order $\leq_{\E_j^\X}$. 
That way, the  edge sequence (and thus the edge structure) remains the same, i.e.~$\E^{\T_w}_j=\E^\X_j$, for all $j=1,\ldots,k-1$.
Moreover consider the same $\E$-edge labelling $\mu^{\T_w}:=\mu$, as for $\X$.
\item for any $j=1,\ldots,k$, the vertex poset $(\V^{\T_w}_j,\leq_{\V^{\T_w}_j})$, where\footnote{if $f(r_i)\neq a_j$, for all $j$, then the right component of the direct sum is the emptyset and we get $\V^{\T_w}_j=\V^\X_j$.}
    $$\V^{\T_w}_j:=\V^\X_j\coprod\biggl(\coprod_{f(r_i)=a_j}\{e^{(i)}_t:1\leq t\leq d_i\text{ and }t\neq w_i\}\biggl),$$
and the partial order $\leq_{\V^{\T_w}_j}$ is given by the identities on the elements, and the additional formal inequalities $r_i<_w e^{(i)}_t$, for all  $1\leq t\leq d_i$ with $t\neq w_i$.
\begin{remark}
These nontrivial innequalities formalize the idea that the vertex $r_i$ is isolated from the `new vertices' (leaf nodes) that are formed  after cutting the reticulation vertex $r_i$ of $\X$ while keeping it connected from above with the edge $e_{w_i}^{(i)}$.
This observation is crucial since, on one hand the trees $\T_w$ are distinguised for all choces of $w$, and on the other hand this makes the coproduct well defined as we will see, e.g.~in Ex.~\ref{ex:3}.
\end{remark}

    \item the function $D_j^{\T_w}:\E_j^{\T_w}\to \V_j^{\T_w}$ given by
    \[
    D_j^{\T_w}(e)=
    \begin{cases}
      D^\X_j(e), & \text{ if } D^\X_j(e)\text{ is not a reticulation-vertex} \\
      r_i, & \text{ if }e=e_{w_i}^{(i)}\text{, for some }i=1,\ldots,m,\\
      e, &\text{ if }D^\X_j(e)=r_i\text{ and }e\neq e_{w_i}^{(i)}\text{, for all }i=1,\ldots,m.
       \end{cases}
   \]
    \item the function   $U_j^{\T_w}:\E_j^{\T_w}\to\V^{\T}_{j+1}$  given by 
    \end{itemize}
    \[
    U_j^{\T}(e)=U^\X_j(e).
   \]
Note that since the edge sets are trivially partially ordered, as such, these functions are trivially order preserving, i.e.~$D_j^{\T_w}:(\E_j^{\T_w},\leq_{\E_j^{\T_w}})\to (\V_j^{\T_w},\leq_{\V^{\T_w}_j})$ and  $U_j^{\T_w}:(\E_j^{\T_w},\leq_{\E_j^{\T_w}})\to(\V^{\T}_{j+1},\leq_{\V^{\T_w}_{j+1}})$ for all $j$.

\vspace{1em}
\noindent 2. $\T_{w}$ is the quotient of the disjoint union 
$$\biggl(\coprod_{j=1}^{k}\V_j^{\T_w}\times\{a_j\}\biggl)\coprod\biggl(\coprod_{j=1}^{k-1}(\E_j^{\T_w}\times [a_i,a_{i+1}]\biggl)$$
with respect to the identifications $(U_j^{\T_w}(e),a_{j+1})\sim (e,a_{j+1})$ and $(D_j^{\T_w}(e),a_{j})\sim (e,a_j)$.
 Define $f_w$ to be the projection of $\T_w$ to the second coordinate. 
Both the disjoint union and the resulting quotient space receive a partial order from the partial orders on  $\V_j^{\T_w}$ and $\E_j^{\T_w}$.
We denote the resulting partial order on $\T_w$ by $\leq_{\T_w}$.
\begin{remark}
Note also that, by definition of $D^{\T_w}$, the identifications  $(D_j^{\T_w}(e),a_{j})\sim (e,a_j)$ are the trivial ones $(e,a_j)\sim (e,a_j)$ whenever  $e\neq e_{w_i}^{(i)}$, for all $i=1,\ldots,m$. 
That formally expresses the fact that the tree $\T_w$ is constructed from $\X$ by cutting the bottom of all cycles, keeping, for all $1\leq i\leq m$, only the edge $e^{(i)}_{w_i}$ connected to the reticulation-vertex $r_i$ and  diconnecting the others. 
\end{remark}

\vspace{1em}
\noindent To sum up, $\T_w$ forms an an ordered  tree $(\T_{w},f_w,\leq_w,\mu)$  in $\mathbf{Reeb}^{ord}[\E_\bullet]$.

\vspace{1em}
\noindent Furthermore, each ordered  tree $(\T_{w},f_w,\leq_{w},\mu)$ is equipped with the obvious quotient map
\begin{align*}
 q_{w}:(\T_{w},f_w,\leq_w,\mu)&\twoheadrightarrow(\X,f,\leq_\X,\mu)\text{, where }\\
q_{w}^{\E}:\E^{\T_{w}}&\to\E^{\X}\\
 e&\mapsto e\text{, and }\\
 q_{w}^{\V}:\V^{\T_{w}}&\to\V^{\X}\\
   v&\mapsto v\text{, for all }v\in\V^\X\\
    e_t^{(i)}&\mapsto r_i \text{, for all }1\leq t\leq d_i\text{ with }t\neq w_i\text{, for all }i=1,\ldots,m.
\end{align*}

By definition, the quotient map $q_w$ glues back the new vertices of the tree $\T_w$--yielded by the $w$-cut--to form the original Reeb graph $\X$.
Note that no matter what the order of the new vertices is, from the cut of $r_i$,  via the quotient  map $q_w$ they are glued to a single vertex $r_i$, hence $q_w$ trivially preserves the partial orders. 
Namely, $q_w$ kills the only non-trivial innequalities $r_i\leq e^{(i)}_t$, for all $t\neq w_i$, and $i=1,\ldots,m$.
Moreover, by construction, $\T_w$ has the same edge sets as $\X$, thus making $q_w$ to be trivially $\E$-labelling preserving.
So indeed $q_w$ is a morphism  in $\Reeb^{ord}[\E_\bullet]$.

\vspace{1em}
\noindent Now we claim that the coproduct of all the trees $(\T_w,f_w,\leq_w,\mu)$ in $\mathbf{Reeb}^{ord}[\E_\bullet]$ is isomorphic to the ordered Reeb graph $(\X,f,\leq_\X,\mu)$ (where, again, $\leq_\X$ is the trivial order of the Reeb graph) with the coprojections morphisms being the quotient maps  $q_w$.

Let $\varphi_{w}:(\T_{w},f_{w},\leq_w,\mu)\to(\Y,g,\leq_\Y,\mu^\Y)$ be any $w$-indexed family of morphisms in $\mathbf{Reeb}^{ord}[\E_\bullet]$.
Then, in particular
\begin{align*}
\varphi_{w}:(\T_{w},f_{w},\leq_w)&\to(\Y,g,\leq_\Y)\\
   [(v,t)]&\mapsto [(\varphi_{w,j}^{\V}(v),t)]\text{, for all  }v\in \V_j^\X\\
 [(e,t)]&\mapsto [(\varphi_{w,j}^{\E}(e),t)]\text{, for all  }e\in \E_j^\X,
\end{align*}
for some order preserving maps $\varphi_{w,j}^{\V}$ and $\varphi_{w,j}^{\E}$ satisfying  the consistency conditions as described in Prop.~\ref{prop:consistency}.
We claim that there exists a unique morphism $\varphi:(\X,f,\leq_\X,\mu)\to(\Y,g,\leq_\Y,\mu_\Y)$ in $\mathbf{Reeb}^{ord}[\E]$ such that  each of the diagrams
\[
\begin{tikzcd}
(\T_{w},f_{w},\leq_w,\mu)\arrow[rrr,two heads, "{q_{w}}"] \arrow[rrrdd,swap,"{\varphi_{w}}"]\&\&\& (\X,f,\leq_\X,\mu)\arrow[dd, dashed,"{\varphi}"]
\\
\\
\&\&\& (\Y,g,\leq_\Y,\mu^\Y)
\end{tikzcd}
\]
commutes.
Since the category $\mathbf{Reeb}^{ord}[\E_\bullet]$ is thin, all diagrams commute by default and if $\varphi$ exists then it is unique. 
Therefore, we only need  to show that there exists a morphism
\begin{align*}
\varphi:(\X,f,\leq_\X,\mu)&\to(\Y,g,\leq_\Y,\mu^\Y)\\
   [(v,t)]&\mapsto [(\varphi_{j}^{\V}(v),t)]\text{, for all  }v\in \V_j^\X\\
 [(e,t)]&\mapsto [(\varphi_{j}^{\E}(e),t)]\text{, for all  }e\in \E_j^\X,
\end{align*}
for some order preserving maps $\varphi_{j}^{\V}$ and $\varphi_{j}^{\E}$ satisfying the consistency conditions as described in Prop.~\ref{prop:consistency}.

Lets return to $\varphi_w$.
Each map $\varphi_{w,j}^{\E}$ is trivially order preserving; the edge posets are just edge sets, and therefore they have the trivial partial order given only by the identities.
Since $\varphi_w$ is $\E$-edge labelling preserving, we get $\varphi_{w,j}^{\E}(e)=(\mu^{\Y})^{-1}\mu(e)$, for all $e\in\E_j^{\T_w}$.
Thus $\varphi_{w,j}^{\E}$ it is independent of $w$.
So, naturally, we define $\varphi^{\E}_j$ by
$$\varphi^{\E}_j(e)=\varphi_{w,j}^{\E}(e)=(\mu^{\Y})^{-1}\mu(e)\text{, for all }e\in\E_j^{\X}=\E_j^{\T_w}.$$

Now lets focus on the order preserving maps $\varphi_{w,j}^{\V}:(\V_j^{\T_w},\leq_{\V_j^{\T_w}})\to(\V_j^\Y,\leq_{\V_j^\Y})$, for $j=1,\ldots,k$.
We compute 
\begin{align*}
\varphi_{w,j}^{\V}(r_i)&=\varphi_{w,j}^{\V}(D^{\T_w}_j(e_{w_i}^{(i)}))\\
&=D^{\Y}_j\varphi_{w,j}^{\E}(e_{w_i}^{(i)})\\
&=D^{\Y}_j\varphi_{j}^{\E}(e_{w_i}^{(i)}),
\end{align*}
 and
\begin{align*}
\varphi_{w,j}^{\V}(e^{(i)}_t)&=\varphi_{w,j}^{\V}(D^{\T_w}_j(e^{(i)}_t))\\
&=D^\Y_j\varphi_{w,j}^{\E}(e^{(i)}_t)\\
&=D^\Y_j\varphi_{j}^{\E}(e^{(i)}_t),
\end{align*}
for all $1\leq t\leq d_i$ with $t\neq w_i$.

On the other hand, the inequalities $r_i\leq_{\V_j^{T_w}} e_t^{(i)}$ yield the innequalities $\varphi_{w,j}^{\V}(r_i)\leq_{\V_j^\Y}\varphi_{w,j}^{\V}(e^{(i)}_t)$.
Therefore 
$$D^{\Y}_j\varphi_{j}^{\E}(e_{w_i}^{(i)})\leq_{\V_j^\Y} D^\Y_j\varphi_{j}^{\E}(e^{(i)}_t).$$
However note that since $w_i$ was chosen in random, we can as well get the inequality
\begin{equation*}
D^\Y_j\varphi_{j}^{\E}(e^{(i)}_t)\leq_{\V_j^\Y} D^{\Y}_j\varphi_{j}^{\E}(e_{w_i}^{(i)}),
\end{equation*}
by formally replacing $w_i$ with $t$.
Therefore we obtain the equalities
\begin{equation}
\label{eqn:independence}
D^\Y_j\varphi_{j}^{\E}(e^{(i)}_t)= D^{\Y}_j\varphi_{j}^{\E}(e_{s}^{(i)})\text{, for all }1\leq t,s\leq d_i\text{, for all }i=1,\ldots,m. 
\end{equation}
Thus we see that
\begin{equation}
\label{eqn:equality}
\varphi_{w,j}^{\V}(r_i)=\varphi_{w,j}^{\V}(e^{(i)}_t)\text{, for all }1\leq t\leq d_i\text{ with }t\neq w_i,
\end{equation}
and, in particular by Eqn.~\ref{eqn:independence}, we see that this is independent of $w$. 
So, naturally, setting $\varphi_{j}^{\V}(r_i):=\varphi_{w,j}^{\V}(r_i)$ is well defined.

Now, for any vertex $v\in \V_j^\X$ which is not a reticulation-vertex, by construction of $\T_w$, we have $v\in \V_j^{\T_w}$.
In particular, $v \in\V_j^{\T_w}$ is either the down image of some $e\in \E_j^{\T_w}$ (via the map $D^{\T_w}$) or the upper  image of some $e\in \E_{j-1}^{\T_w}$ (via the map $U^{\T_w}$).
Because of that, we can, again show that $\varphi_{w,j}^{\V}(v)$ is independent of the choice of $w$.
So we define $\varphi^{\V}_j(v):=\varphi_{w,j}^\V(v)$.
To sum up, $\varphi^{\V}_j(v):=\varphi_{w,j}^\V(v)$  is a well defined map, for all $v\in \V_j^{\X}$.

We claim that the maps $\varphi_j^{\V}$ and $\varphi_j^{\E}$ satisfy the consistency conditions  as described in Prop.~\ref{prop:consistency}.
We only need to check the consisteny condition $\varphi_j^{\V}D^{\X}_j=D^{\Y}_j\varphi_j^{\E}$.
The other one is true since $U^{\T_w}_j(e)=U^{\X}_j(e)$ for all $e\in \E_{j-1}$.
If we show that $\varphi_j^{\V}D^{\X}_j=D^{\Y}\varphi_j^{\E}$, then the map $\varphi:(\X,f)\to(\X,g)$ would be an actual morphism in $\Reeb$.

Indeed, fix any $w$. 
Let $e\in \E_j^{\T_w}$.
We have the following cases:
\begin{itemize}
\item \textbf{Case 1, $D^\X_j(e)$ is not a reticulation-vertex:} We compute
\begin{align*}
\varphi_j^{\V}D^{\X}_j(e)&=\varphi_{w,j}^{\V}D^{\T_w}_j(e)\text{ (by definition of }D^{\T_w}_j\text{ and  because }\E_j^\X=\E_j^{\T_w})\\
&=D^{\Y}\varphi_{w,j}^{\E}(e)\text{ (by consistency of }\varphi_{w,j}^{\E})\\
&=D^{\Y}\varphi_j^{\E}(e).
\end{align*}
\item \textbf{Case 2, $D^\X_j(e)=r_i$ and $e=e_{w_i}^{(i)}$:}
We compute
\begin{align*}
\varphi_j^{\V}D^{\X}_j(e_{w_i}^{(i)})&=\varphi_j^{\V}(r_i)\\
&=\varphi_{w,j}^{\V}(r_i)\\
&=\varphi_{w,j}^{\V}(D^{\T_w}_j(e_{w_i}^{(i)}))\text{ (by definition of }D^{\T_w}_j)\\
&=D^{\Y}\varphi_{w,j}^{\E}(e_{w_i}^{(i)})\text{ (by consistency of }\varphi_{w,j}^{\E})\\
&=D^{\Y}\varphi_{j}^{\E}(e_{w_i}^{(i)}).
\end{align*}
\item \textbf{Case 3, $D^\X_j(e)=r_i$ and $e\neq e_{w_i}^{(i)}$:}
We compute
\begin{align*}
\varphi_j^{\V}D^{\X}_j(e)&=\varphi_j^{\V}(r_i)\\
&=\varphi_{w,j}^{\V}(r_i)\\
&=\varphi_{w,j}^{\V}(e)\text{ (by Eqn.~(\ref{eqn:equality}))}\\
&=\varphi_{w,j}^{\V}D^{\T_w}_j(e)\text{ (by definition of }D^{\T_w}_j)\\
&=D^{\Y}\varphi_{w,j}^{\E}(e)\text{ (by consistency of }\varphi_{w,j}^{\E})\\
&=D^{\Y}\varphi_{j}^{\E}(e).
\end{align*}
\end{itemize}
Hence there exists a  map $\varphi:(\X,f)\to(\X,g)$ in $\Reeb$.
Note that since $\X=(\X,f)$ is trivially partially ordered, then $\varphi_w$ is trivially order preserving.
Finally, by definition, we have $\varphi_j^{\E}=(\mu^\Y)^{-1}\circ\mu$.
Hence $\varphi$ is $\E$-edge labelling preserving.
Therefore $\varphi$ is a morphism in $\Reeb^{ord}[\E_\bullet]$.

\begin{remark}
\label{rem:remark}
By definition of $\T_w$, each of the leaves of $\T_w$ that have been created after the $w$-cut is identified with an edge 
that is incident from above of a reticulation-vertex but is not equal to $e^{(i)}_{w_i}$ for all $i=1,\ldots,m$.
Therefore the cardinality of the set of the leaves that are additional to the $n$-leaves of $\X$ is exactly
$$\sum_{i=1}^{m}(d_i-1),$$
which is equal to the first Betti number $B_1=s$ of $\X$, because of Prop.~\ref{prop:betti}.
Hence, each  tree $\T_{w}$ has $(n+s)$-leaves.
\end{remark}
\end{proof}


\begin{example}
\label{ex:2}
Consider a  Reeb graph $\X=(\X,f)$ with four leaves and one cycle, and edges labelled by $e_1,e_2,\ldots,e_8$, as shown in  Fig.~2.
By applying Thm.~\ref{thm:ClassReeb} we get two trees with $5$-leaves in the tree-decomposition of $\X$.
For each of the trees, the one additional leaf to the four leaves of $\X$ corresponds to one of the two  edges $e_5$ and $e_6$, respectively.
Thus there are exactly two such trees with $5$-leaves.
\begin{figure}[h!]
\label{fig:simple}
  \begin{center}
      \includegraphics[width=0.9\textwidth]{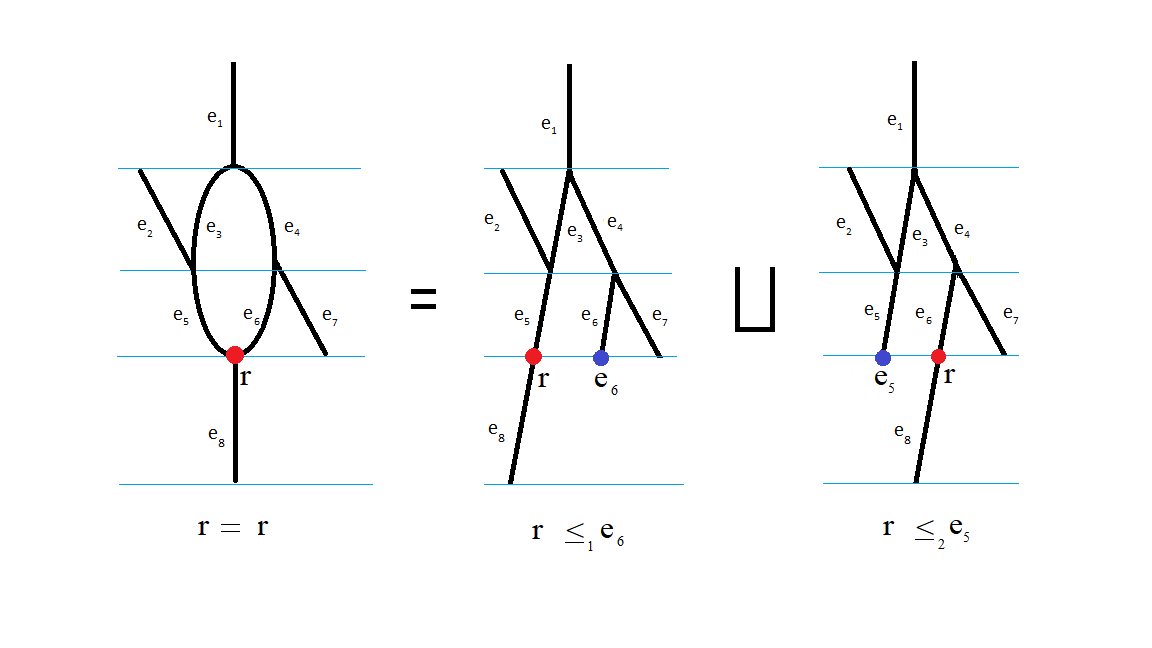}
   \end{center}
   \caption{An example describing the tree-decomposition for $n=4$ and $s=1$.}
\end{figure}
\end{example}

\begin{example}
\label{ex:3}
Consider a  Reeb graph $\X=(\X,f)$ with a single reticulation-vertex--where in this particular example we also think of it as a leaf--three edges, and two cycles, and edges labelled by $e_1,e_2,e_3$, as shown in  Fig.~3.
By applying Thm.~\ref{thm:ClassReeb} we get three trees with $4$-leaves in the tree-decomposition of $\X$.
Note that the trees are in one to one correspondece with all possibilities of isolating an edge (which is identified with the reticulation-vertex) from the other edges, given that they are incident from above some reticulation vertex.  
Thus there are exactly three such trees with $4$-leaves.

\begin{figure}[h!]
\label{fig:tree}
  \begin{center}
      \includegraphics[width=0.8\textwidth]{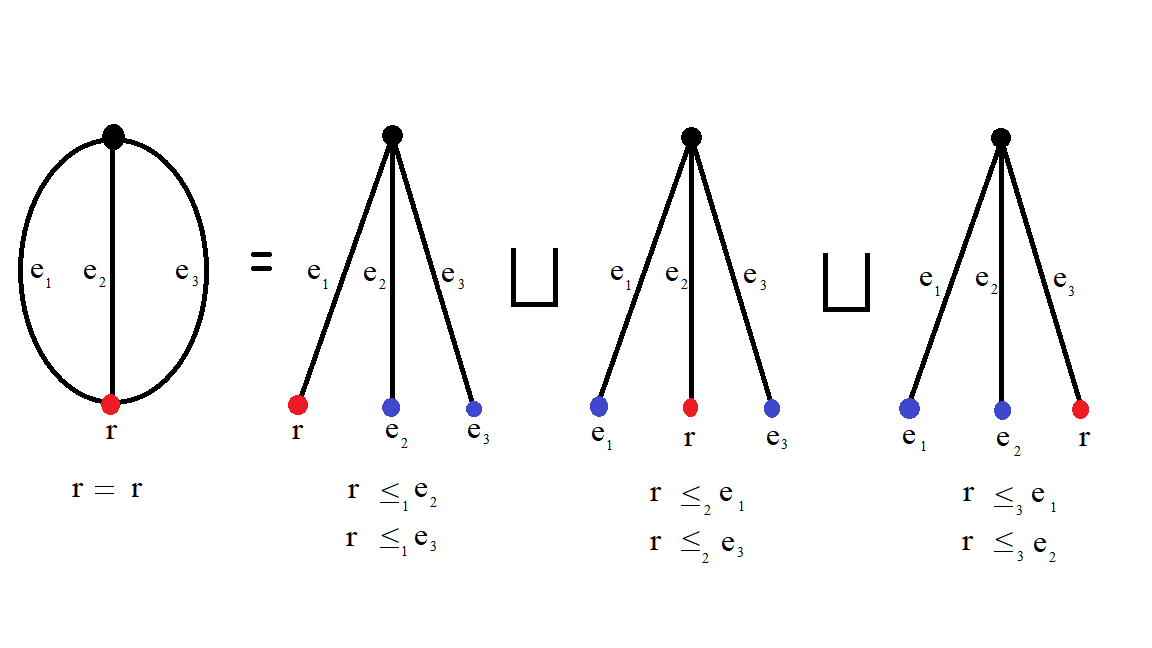}
   \end{center}
   \caption{An example describing the tree-decomposition for $n=1$ and $s=2$.}
\end{figure}
\end{example}

\section{Isomorphism complexity of Reeb graphs}
Thm.~\ref{thm:ClassReeb} can help  improve our understanding of the isomorphism complexity of Reeb graphs  under some fixed parameter.
Here we will refer to \textit{isomorphism complexity} simply by \textit{complexity}.

\subsection{Upper bound on the complexity} Let $\X=(\X,f)$ be a Reeb graph with edge-sequence $\E$. 
Let $\Y=(\Y,g)$ any Reeb graph.
Checking whether $\Y$ has same edge-structure as $\X$ is equivalent to checking whether $\Y$ has equivalent edge-sequence with $\E$.
This takes time $O(|\E|)$, where $|\E|$ is the cardinality of $\E$.
Now, by Euler's characteristic formula we have $s=|\E|-|\V|+1$, and so $|\E|=|\V|+s-1$.
By definition, if two Reeb graphs $(\X,f),(\Y,g)$ are isomorphic, then they have the same edge structure.
Therefore:
$$\biggl(\text{Complexity of Reeb graphs}\biggl)=\biggl(\text{Complexity of Reeb graphs of same edge structure}\biggl) + O(|\V|+s).$$  
By Thm.~\ref{thm:ClassReeb}, two Reeb graphs of same edge structure, with $n$ leaves and Betti number $B_1(\X)=s$, are isomorphic if and only if they have the same tree-decomposition.
To check for isomorphism it takes at worst, as many number of steps as the cardinality of the set $\TT(\X)$ times the complexity of checking whether two ordered trees with $(n+s)$-leaves are isomorphic in $\Reeb^{ord}[\E_\bullet]$, where $\E_\bullet$ is the edge structure of $\X$. 
Namely:
$$\biggl(\text{Complexity of Reeb graphs of same edge structure}\biggl)=\biggl(\text{Complexity of ordered trees}\biggl)\cdot O(|\TT(\X)|).$$
Now checking if two trees with $(n+s)$-leaves are isomorphic takes time $O(n+s)$.\footnote{It is known from basic graph theory that checking  if two trees with $k$-leaves are isomorphic takes $O(k)$ steps }.
By construction, the cardinality of $\TT(\X)$ of a Reeb graph $\X$ is equal to the product of indegrees of the reticulation-vertices,
i.e.
$$|\mathcal{T}(\X)|=\prod_{i=1}^{m}d_i$$
Therefore we get
$$\biggl(\text{Complexity of Reeb graphs}\biggl)=O((n+s)\prod_{i=1}^{m}d_i)+O(|\V|+s).$$
Now because $L(\X)\subset \V$ and $|L(\X)|=n$, we get $n\leq |\V|$.
Also since $s\geq 0$ ($s=0$ if $\X$ is a tree) and $|\V|\geq 2$, we have the obvious bounds $n+s\leq |\V|+s\leq |\V|(1+s)$.
Thus:
$$\biggl(\text{Complexity of Reeb graphs}\biggl)=O(|\V| (1+s)\prod_{i=1}^{m}d_i)+O(|\V|(1+s))=O(|\V| (1+s)\prod_{i=1}^{m}d_i).$$

\subsection{Parametrized complexity}
Although this is a good bound, we would like to consider an upper bound on the isomorphism complexity that does not depend on the indegrees of reticulation-vertices, but  only depends on the number of vertices $|\V|$ and the Betti number $B_1(\X)=s$.
Since each of the indegrees of the reticulation nodes is $d_i>1$ we have the bound $d_i\leq 2^{d_i-1}$.
Taking the product over all these inequalities and since $$\sum_{i=1}^{m}(d_i-1)=s,$$
we get
$$\biggl(\text{Complexity of Reeb graphs with $B_1(\X)=s$}\biggl)=O(|\V| (1+s) 2^{s}).$$
Hence, the Reeb graph isomorphism problem is fixed parameter tractable when the parameter is the first Betti number of the graph.
Finally, note that the  bound $d_i\leq 2^{d_i-1}$ is tight: indeed, if we consider Reeb graphs with $n$-leaves and  $B_1(\X)=s$, such that the indegree of each of its reticulation-vertices is $2$, then the product of all the indegrees is exactly $2^s$.

\section{Phylogenetic networks viewed as  ordered Reeb graphs}
In computational phylogenetics a \textit{phylogenetic network} is any rooted DAG, $G$, that can be used to represent the evolutionary relationships among biological organisms, e.g.~genes, often called \textit{taxa}. 
These taxa are represented by an ordering on the leaves of the rooted DAG.
\begin{remark}
Quite often in practice, the taxa are represented by the totaly ordered set $\{1<2<\ldots<n\}$. 
However, this is very restrictive for general phylogenetic networks. 
Namely, one can have a phylogenetic network where some of its leaves are ordered in many different ways, and some of them might not be labelled at all.
So it is better to consider a partial order on the leaves.
\end{remark}
Naturally one considers a pair of phylogenetic networks to be \textit{isomorphic}  if there is an isomorphism between their underlying DAGs that also preserves the \textit{order}  of their corresponding taxa.

\subsection{Time consistent phylogenetic networks}
In many applications, phylogenetic networks are equipped with a \textit{time assignment} on the nodes or vertices, which we can think of as a height function $f:G\to \R$ on the graph $G$.
The time assignment map $f$ is often required to be \textit{time consistent} which essentially means that the time stamp of any `non-root node' should be strictly larger than the time stamp of its `parent node'.
See \cite{huson2010phylogenetic}, pg 167. 
This is exactly what Reeb graphs satisfy as we emphasized in Ex~\ref{ex:counter}.
Moreover, if we restrict to ordered 
Reeb graphs, then these structures
can be thought of as vertex weighted DAGs (see Sec.~\ref{sec:vertex-weighted}) such that their edge sets and vertex sets are partially ordered in a compative way.
So we believe that ordered rooted Reeb graphs whould serve as a natural mathematical model for time consistent phylogenetic networks.  
By thinking of a time consistent phylogenetic network as an ordered Reeb graph, we can also consider its tree decomposition, given by Thm.~\ref{thm:ClassReeb}.
Since the phylogenetic network is rooted and has its leaves partially ordered, by construction, each of the trees $\T_w$, yielded by the network's tree decomposition, would have their leaves partially ordered.
Therefore, each tree $\T_w$ is a phylogenetic tree.

\subsection{Hausdorff distance} 
Among many distance metrics, we can utilize the $\ell^p$-cophenetic metrics, $d_p$, for $1\leq p\leq \infty$, for comparing a pair of   phylogenetic trees \cite{cardona2013cophenetic,munch2018ell}.
Lets focus on $d_\infty$ (for $p=\infty$). 
The idea is that, via the cophenetic map \cite{sokal1962comparison}, any phylogenetic tree with $n$ ordered leaves, $\{1<\ldots<n\}$, can be identified with a single point in the $[n(n+1)/2]$-dimensional Euclidean space.
The $(i,j)$-th coordinate of the point, where $i\leq j$, corresponds to the time-stamp of the \textit{least common ancestor} of the leaves  $i,j$ in the phylogenetic tree \cite{munch2018ell}.
Then, we simply consider the $\ell^\infty$-norm for comparing a pair of phylogenetic trees with $n$ ordered leaves. 
By the tree decompostion of Thm.~\ref{thm:ClassReeb}, we can thus identify any time consistent phylogenetic network with $n$ ordered leaves, and first Betti number $s$, as a finite subset  of the the $[(n+s)(n+s+1)/2]$-dimensional Euclidean space.  
Hence, we can utilize the Hausdorff distance for comparing a pair of phylogenetic networks. 
\begin{definition}
Let $(X,d_X)$ be any metric space, and let $A,B$ be any subsets of $X$.
The \textit{Hausdorff distance of $A,B$} is 
$$d_H(A,B)=\max\{\max_{a\in A}\min_{b\in B}d_X(a,b),\max_{b\in B}\min_{a\in A} d_X(a,b)\}.$$ 
\end{definition}

\begin{example}
Consider the pair $N_1,N_2$ of time consistent phylogenetic networks as in Fig.~4, modelled as ordered Reeb graphs, and having the non-trivial partial order on their leaves, $1<2$.
\begin{figure}[h!]
\label{fig:tree}
  \begin{center}
      \includegraphics[width=0.8\textwidth]{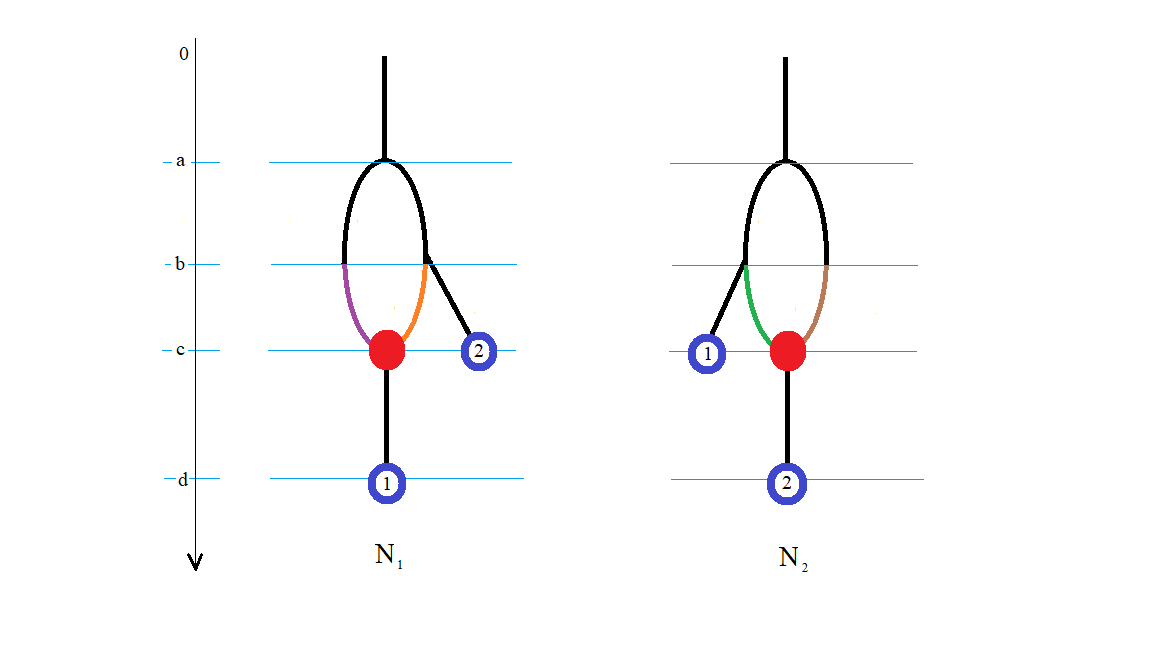}
   \end{center}
   \caption{A pair of time consistent phylogenetic networks, each having $2$ leaves and $1$ cycle.}
\end{figure}
\begin{figure}[h!]
\label{fig:tree}
  \begin{center}
      \includegraphics[width=0.8\textwidth]{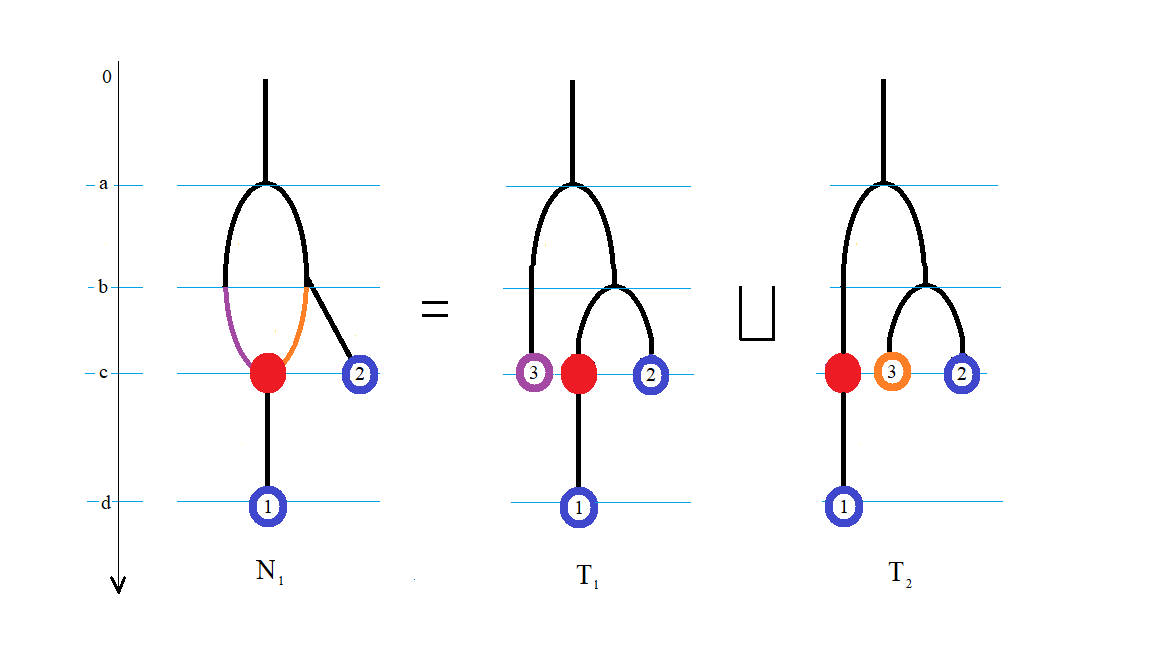}
   \end{center}
\end{figure}
\begin{figure}[h!]
\label{fig:tree}
  \begin{center}
      \includegraphics[width=0.8\textwidth]{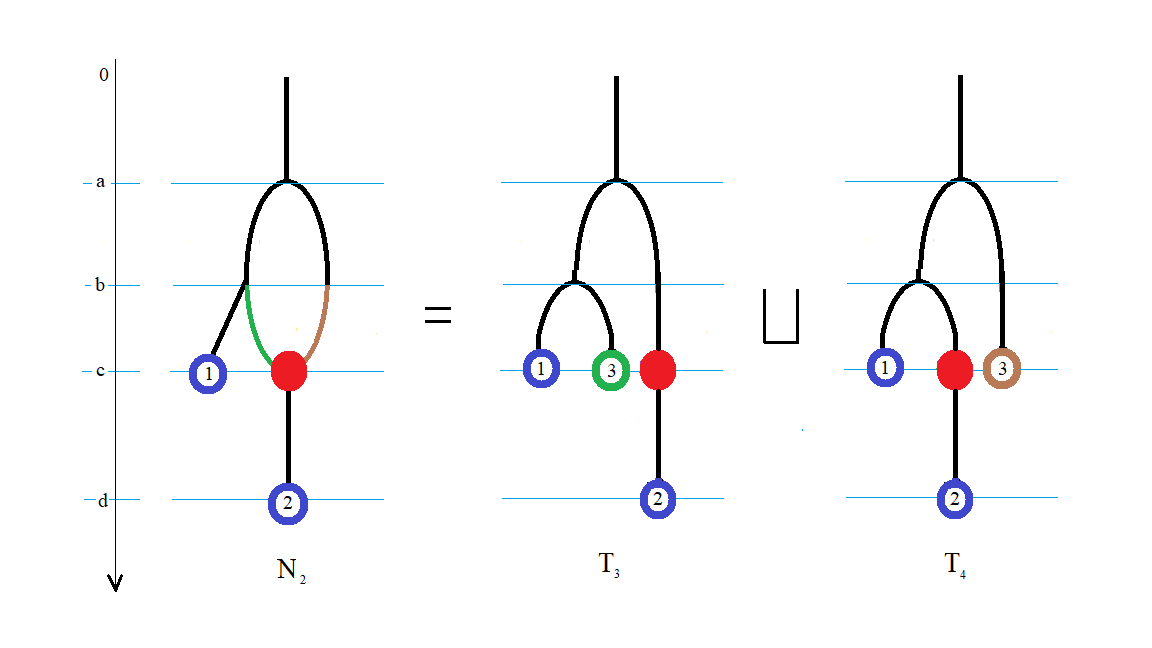}
   \end{center}
   \caption{The additional leaf in each tree corresponds to an edge incident from above the reticulation vertex.}
\end{figure}

Also assume $0<a<b<c<d$ and $c-b<b-a<d-c$ as in Fig.~4.
It is easy to check that, if we forget the non-trivial order on the leaves of $N_1,N_2$, then these networks are isomorphic as ordinary Reeb graphs.
However they are not order-preserving isomorphic.
We do this by showing that the Hausdorff distance--with respect to $d_\infty$--of the tree decompositions of $N_1,N_2$ is non-zero.
Let $\TT(N_1)=\{T_1,T_2\}$ and $\TT(N_2)=\{T_3,T_4\}$ be the tree decompositions of $N_1$ and $N_2$ as in Fig.~5.
We have 
\begin{align*}
d_\infty(T_1,T_3)&=\biggl|\biggl|
\left( {\begin{array}{ccc}
   d & b & a\\
   \cdot & c & a\\
\cdot & \cdot& c\\
  \end{array} } \right)-
\left( {\begin{array}{ccc}
    c & a & b\\
   \cdot & d & a\\
\cdot & \cdot & c\\
  \end{array} } \right)\biggl|\biggl|_\infty= d-c,\\
d_\infty(T_1,T_4)&=\biggl|\biggl|
\left( {\begin{array}{ccc}
     d & b & a\\
   \cdot & c & a\\
\cdot & \cdot& c\\
  \end{array} } \right)-
\left( {\begin{array}{ccc}
   c & b & a\\
   \cdot & d & a\\
\cdot & \cdot & c\\
  \end{array} } \right)\biggl|\biggl|_\infty =d-c,\\
d_\infty(T_2,T_3)&=\biggl|\biggl|
\left( {\begin{array}{ccc}
     d & a & a\\
   \cdot & c & b\\
\cdot & \cdot& c\\
  \end{array} } \right)-
\left( {\begin{array}{ccc}
    c & a & b\\
   \cdot & d & a\\
\cdot & \cdot & c\\
  \end{array} } \right)\biggl|\biggl|_\infty=d-c,\\
d_\infty(T_2,T_4)&=\biggl|\biggl|
\left( {\begin{array}{ccc}
     d & a & a\\
   \cdot & c & b\\
\cdot & \cdot & c\\
  \end{array} } \right)-
\left( {\begin{array}{ccc}
   c & b & a\\
   \cdot & d & a\\
\cdot & \cdot & c\\
  \end{array} } \right)\biggl|\biggl|_\infty=d-c,\\
\end{align*}
since $c-b<b-a<d-c$.
Hence, we obtain
\begin{align*}
d_H(\TT(N_1),\TT(N_2))&=\max\biggl\{\max_{1\leq i\leq 2}\bigl\{\min_{3\leq j\leq4}d_\infty(T_i,T_j)\bigl\},\max_{3\leq j\leq 4}\bigl\{\min_{1\leq i\leq 2}d_\infty(T_i,T_j)\bigl\}\biggl\}\\
&=d-c>0.
\end{align*}

\end{example}

\begin{remark}
\label{rem:hausdorff complexity}
Assume $A,B$ are finite. 
By definition, the Hausdorff distance between $A,B$ can be computed in at most $|A|\cdot |B|$-many steps.
That means, in particular, that computing the Hausdorff distance of a pair of phylogenetic networks with $n$ leaves and first Betti number $s$, takes $O(|\V|^2 (1+ s)^2 4^{s})$ time.
Namely, computing the Hausdorff distance is fixed parameter tractable when the parameter is the first Betti number.
\end{remark}

\section{Concluding remarks}
In this manuscript we showed that   Reeb graphs are classified up to isomorphism by their tree-decomposition, we constructed upper bounds for their isomorphism complexity, and showed that the isomorphism problem for Reeb graphs is fixed parameter tractable when the parameter is the first Betti number.
We proposed $\Reeb^{ord}$ as model for phylogenetic networks.
Moreover, we proposed the use of Hausdorff distance as a metric for phylogenetic networks with $n$ leaves and first Betti number $s$.
We speculate that our results, on one hand would further our understanding of both the structure and isomorphism complexity of Reeb graphs, and on the other hand they will enhance the existing methods on phylogenetic networks by providing new insights on how to do statistics and data analysis on these structures.
 
\textbf{Future work:} It is in the author's interests to apply properly the notion of Hausdorff distance to define a metric on time consistent phylogenetic networks. 
In order to do that some restrictions on the ordered Reeb graphs (phylogenetic networks) may need to be considered in addition to requiring a fixed number of leaves and cycles, e.g.~we may need to assume in particular a total order on the leaves and reticulation vertices.
Also a natural question to ask is whether there exist other tree decompositions for general ordered Reeb graphs, perhaps with fewer tree factors, and how do they look like.

Furthermore, just like Reeb graphs can be identified with nice enough cosheaves on $\R$ valued in the category $\Set$ of all sets \cite{de2016categorified}, it  seems that ordered Reeb graphs can be identified with constructible cosheaves $\Ffunc:\Int\to\mathbf{Pos}$ on $\R$ valued in the category $\mathbf{Pos}$ of all posets.
If this speculation is true then one can define a notion of \textit{interleaving distance} for ordered Reeb graphs--in the sense of Bubenik et al.~\cite{bubenik2015metrics}--and thus phylogenetic networks in particular.
The interleaving distance will thus be a metric on the collection of arbitrary phylogenetic networks, i.e.~not just the ones with fixed number of leaves and first Betti number.
This would give an advantage for one to use the interleaving metric over the Hausdorff distance as a metric for phylogenetic networks.
By Rem.~\ref{rem:hausdorff complexity}, computing the Hausdorff distance of a pair of phylogenetic networks is fixed paremeter tractable when the parameter is the first Betti number, but we do not know whether this is the case also for the interleaving distance. 
In the near future the author wishes to work towards the sheaf-theoretic aspects of ordered Reeb graphs and their implications to phylogenetics, e.g~defining an intelreaving metric for the comparison of arbitrary time consistent phylogenetic networks.

\section{Conflict of Interest Statement}
The author states that there is no conflict of interest.

\bibliographystyle{plain}  

\begin{thebibliography}{10}

\bibitem{adamek1990herrlich}
JIRI Ad{\'a}mek.
\newblock Herrlich and h., strecker, ge, abstract and concrete categories. the
  joy of cats.
\newblock {\em Pure and Applied Mathematics, A Wiley-Interscience Publication.
  John Wiley \& Sons, Inc., New York, xiv}, 1990.

\bibitem{agarwal2006extreme}
Pankaj~K Agarwal, Herbert Edelsbrunner, John Harer, and Yusu Wang.
\newblock Extreme elevation on a 2-manifold.
\newblock {\em Discrete \& Computational Geometry}, 36(4):553--572, 2006.

\bibitem{bauer_et_al:LIPIcs:2015:5146}
Ulrich Bauer, Elizabeth Munch, and Yusu Wang.
\newblock {Strong Equivalence of the Interleaving and Functional Distortion
  Metrics for Reeb Graphs}.
\newblock In Lars Arge and J{\'a}nos Pach, editors, {\em 31st International
  Symposium on Computational Geometry (SoCG 2015)}, volume~34 of {\em Leibniz
  International Proceedings in Informatics (LIPIcs)}, pages 461--475, Dagstuhl,
  Germany, 2015. Schloss Dagstuhl--Leibniz-Zentrum fuer Informatik.

\bibitem{biasotti2008reeb}
Silvia Biasotti, Daniela Giorgi, Michela Spagnuolo, and Bianca Falcidieno.
\newblock Reeb graphs for shape analysis and applications.
\newblock {\em Theoretical computer science}, 392(1-3):5--22, 2008.

\bibitem{billera2001geometry}
Louis~J Billera, Susan~P Holmes, and Karen Vogtmann.
\newblock Geometry of the space of phylogenetic trees.
\newblock {\em Advances in Applied Mathematics}, 27(4):733--767, 2001.

\bibitem{bjerkevik_et_al:LIPIcs:2018:8726}
Håvard~Bakke Bjerkevik and Magnus~Bakke Botnan.
\newblock {Computational Complexity of the Interleaving Distance}.
\newblock In Bettina Speckmann and Csaba~D. T{\'o}th, editors, {\em 34th
  International Symposium on Computational Geometry (SoCG 2018)}, volume~99 of
  {\em Leibniz International Proceedings in Informatics (LIPIcs)}, pages
  13:1--13:15, Dagstuhl, Germany, 2018. Schloss Dagstuhl--Leibniz-Zentrum fuer
  Informatik.

\bibitem{bouland2012tractable}
Adam Bouland, Anuj Dawar, and Eryk Kopczy{\'n}ski.
\newblock On tractable parameterizations of graph isomorphism.
\newblock In {\em International Symposium on Parameterized and Exact
  Computation}, pages 218--230. Springer, 2012.

\bibitem{bubenik2015metrics}
Peter Bubenik, Vin De~Silva, and Jonathan Scott.
\newblock Metrics for generalized persistence modules.
\newblock {\em Foundations of Computational Mathematics}, 15(6):1501--1531,
  2015.

\bibitem{cardona2013cophenetic}
Gabriel Cardona, Arnau Mir, Francesc Rossell{\'o}, Lucia Rotger, and David
  S{\'a}nchez.
\newblock Cophenetic metrics for phylogenetic trees, after {S}okal and {R}ohlf.
\newblock {\em BMC bioinformatics}, 14(1):3, 2013.

\bibitem{cardona2008extended}
Gabriel Cardona, Francesc Rossell{\'o}, and Gabriel Valiente.
\newblock Extended newick: it is time for a standard representation of
  phylogenetic networks.
\newblock {\em BMC bioinformatics}, 9(1):532, 2008.

\bibitem{carlsson2005persistence}
Gunnar Carlsson, Afra Zomorodian, Anne Collins, and Leonidas~J Guibas.
\newblock Persistence barcodes for shapes.
\newblock {\em International Journal of Shape Modeling}, 11(02):149--187, 2005.

\bibitem{chazal2013gromov}
Fr{\'e}d{\'e}ric Chazal and Jian Sun.
\newblock Gromov-hausdorff approximation of metric spaces with linear
  structure.
\newblock {\em arXiv preprint arXiv:1305.1172}, 2013.

\bibitem{cohen2009extending}
David Cohen-Steiner, Herbert Edelsbrunner, and John Harer.
\newblock Extending persistence using poincar{\'e} and lefschetz duality.
\newblock {\em Foundations of Computational Mathematics}, 9(1):79--103, 2009.

\bibitem{cole2004loops}
Kree Cole-McLaughlin, Herbert Edelsbrunner, John Harer, Vijay Natarajan, and
  Valerio Pascucci.
\newblock Loops in reeb graphs of 2-manifolds.
\newblock {\em Discrete \& Computational Geometry}, 32(2):231--244, 2004.

\bibitem{crawley2015decomposition}
William Crawley-Boevey.
\newblock Decomposition of pointwise finite-dimensional persistence modules.
\newblock {\em Journal of Algebra and its Applications}, 14(05):1550066, 2015.

\bibitem{de2016categorified}
Vin De~Silva, Elizabeth Munch, and Amit Patel.
\newblock Categorified reeb graphs.
\newblock {\em Discrete \& Computational Geometry}, 55(4):854--906, 2016.

\bibitem{dey2013efficient}
Tamal~K Dey, Fengtao Fan, and Yusu Wang.
\newblock An efficient computation of handle and tunnel loops via reeb graphs.
\newblock {\em ACM Transactions on Graphics (TOG)}, 32(4):32, 2013.

\bibitem{di2016edit}
Barbara Di~Fabio and Claudia Landi.
\newblock The edit distance for reeb graphs of surfaces.
\newblock {\em Discrete \& Computational Geometry}, 55(2):423--461, 2016.

\bibitem{dress2007category}
Andreas Dress.
\newblock The category of x-nets.
\newblock In {\em Networks: from biology to theory}, pages 3--22. Springer,
  2007.

\bibitem{edelsbrunner2008reeb}
Herbert Edelsbrunner, John Harer, and Amit~K Patel.
\newblock Reeb spaces of piecewise linear mappings.
\newblock In {\em Symposium on Computational Geometry}, pages 242--250, 2008.

\bibitem{escolano2013complexity}
Francisco Escolano, Edwin~R Hancock, and Silvia Biasotti.
\newblock Complexity fusion for indexing reeb digraphs.
\newblock In {\em International Conference on Computer Analysis of Images and
  Patterns}, pages 120--127. Springer, 2013.

\bibitem{ge2011data}
Xiaoyin Ge, Issam~I Safa, Mikhail Belkin, and Yusu Wang.
\newblock Data skeletonization via reeb graphs.
\newblock In {\em Advances in Neural Information Processing Systems}, pages
  837--845, 2011.

\bibitem{harvey2010randomized}
William Harvey, Yusu Wang, and Rephael Wenger.
\newblock A randomized o (m log m) time algorithm for computing reeb graphs of
  arbitrary simplicial complexes.
\newblock In {\em Proceedings of the twenty-sixth annual symposium on
  Computational geometry}, pages 267--276. ACM, 2010.

\bibitem{hilaga2001topology}
Masaki Hilaga, Yoshihisa Shinagawa, Taku Kohmura, and Tosiyasu~L Kunii.
\newblock Topology matching for fully automatic similarity estimation of 3d
  shapes.
\newblock In {\em Proceedings of the 28th annual conference on Computer
  graphics and interactive techniques}, pages 203--212. ACM, 2001.

\bibitem{huson2010phylogenetic}
Daniel~H Huson, Regula Rupp, and Celine Scornavacca.
\newblock {\em Phylogenetic networks: concepts, algorithms and applications}.
\newblock Cambridge University Press, 2010.

\bibitem{lokshtanov2017fixed}
Daniel Lokshtanov, Marcin Pilipczuk, Micha{\l} Pilipczuk, and Saket Saurabh.
\newblock Fixed-parameter tractable canonization and isomorphism test for
  graphs of bounded treewidth.
\newblock {\em SIAM Journal on Computing}, 46(1):161--189, 2017.

\bibitem{mac2013categories}
Saunders Mac~Lane.
\newblock {\em Categories for the working mathematician}, volume~5.
\newblock Springer Science \& Business Media, 2013.

\bibitem{morozov2013interleaving}
Dmitriy Morozov, Kenes Beketayev, and Gunther Weber.
\newblock Interleaving distance between merge trees.
\newblock {\em Discrete and Computational Geometry}, 49(22-45):52, 2013.

\bibitem{munch2018ell}
Elizabeth Munch and Anastasios Stefanou.
\newblock The $\ell^\infty$-cophenetic metric for phylogenetic trees as an
  interleaving distance.
\newblock {\em arXiv preprint arXiv:1803.07609}, 2018.

\bibitem{nicolau2011topology}
Monica Nicolau, Arnold~J Levine, and Gunnar Carlsson.
\newblock Topology based data analysis identifies a subgroup of breast cancers
  with a unique mutational profile and excellent survival.
\newblock {\em Proceedings of the National Academy of Sciences},
  108(17):7265--7270, 2011.

\bibitem{reeb1946points}
Georges Reeb.
\newblock Sur les points singuliers d'une forme de pfaff completement
  integrable ou d'une fonction numerique [on the singular points of a
  completely integrable pfaff form or of a numerical function].
\newblock {\em Comptes Rendus Acad. Sciences Paris}, 222:847--849, 1946.

\bibitem{semple2003phylogenetics}
Charles Semple, Mike~A Steel, Richard~A Caplan, Mike Steel, et~al.
\newblock {\em Phylogenetics}, volume~24.
\newblock Oxford University Press on Demand, 2003.

\bibitem{singh2007topological}
Gurjeet Singh, Facundo M{\'e}moli, and Gunnar~E Carlsson.
\newblock Topological methods for the analysis of high dimensional data sets
  and 3d object recognition.
\newblock In {\em SPBG}, pages 91--100, 2007.

\bibitem{sokal1962comparison}
Robert~R Sokal and F~James Rohlf.
\newblock The comparison of dendrograms by objective methods.
\newblock {\em Taxon}, pages 33--40, 1962.

\bibitem{wood2004removing}
Zo{\"e} Wood, Hugues Hoppe, Mathieu Desbrun, and Peter Schr{\"o}der.
\newblock Removing excess topology from isosurfaces.
\newblock {\em ACM Transactions on Graphics (TOG)}, 23(2):190--208, 2004.

\bibitem{yamazaki1997isomorphism}
Koichi Yamazaki, Hans~L Bodlaender, Babette De~Fluiter, and Dimitrios~M
  Thilikos.
\newblock Isomorphism for graphs of bounded distance width.
\newblock In {\em Italian Conference on Algorithms and Complexity}, pages
  276--287. Springer, 1997.

\bibitem{yao2009topological}
Yuan Yao, Jian Sun, Xuhui Huang, Gregory~R Bowman, Gurjeet Singh, Michael
  Lesnick, Leonidas~J Guibas, Vijay~S Pande, and Gunnar Carlsson.
\newblock Topological methods for exploring low-density states in biomolecular
  folding pathways.
\newblock {\em The Journal of chemical physics}, 130(14):04B614, 2009.

\end{thebibliography}

\end{document}